\documentclass[a4paper, 12pt]{amsart}

\usepackage{verbatim}
\usepackage{amssymb}
\usepackage{amsbsy}
\usepackage{amscd}
\usepackage{amsmath}
\usepackage{amsthm}
\usepackage[mathscr]{eucal}
\usepackage{mathrsfs}  
\usepackage{bm}

\usepackage{comment}

\newtheorem{theorem}{Theorem}\numberwithin{theorem}{section}

\newtheorem{defn}[theorem]{Definition}

\newtheorem{prop}[theorem]{Proposition}

\newtheorem{cor}[theorem]{Corollary}

\newtheorem{rem}[theorem]{Remark}

\numberwithin{equation}{section}

\def\Hom{\operatorname{Hom}}  
 
\def\rank{\operatorname{rank}} 
\def\Res{\operatorname{Res}}

\def\L{\mathbb L}

\def\x{\bm{x}} 
\def\y{\bm{y}}

\newcommand{\Q}{\mathbb{Q}}
\newcommand{\Z}{\mathbb{Z}}

\newcommand{\C}{\mathbb{C}}

\newcommand{\F}{\mathcal{F}}

\newcommand{\hooklongrightarrow}{\lhook\joinrel\longrightarrow}

\title
{
  Darondeau--Pragacz formulas  
 in complex cobordism
}

\author{Masaki Nakagawa    and Hiroshi Naruse}

\address{Graduate School of  Education \endgraf
                       Okayama University \endgraf
                       Okayama  700-8530 \\ Japan 
}
\email{nakagawa@okayama-u.ac.jp}

\address{Graduate School of Education \endgraf 
                        University of Yamanashi   \endgraf
                       Kofu  400-8510 \\ Japan
}

\email{hnaruse@yamanashi.ac.jp} 

\subjclass[2010]{05E05,   14M15, 14N15,  55N20, 55N22, 57R77
                      }

\keywords{Complex cobordism theory, Gysin map, Push-forward, Flag bundles}

\begin{document}

\begin{abstract}
       In this paper, we generalize  the push-forward (Gysin) formulas 
       for flag bundles  in ordinary cohomology theory, which are 
      due to Darondeau--Pragacz,  to the complex cobordism theory. 
      Then,  we   introduce the {\it universal 
      quadratic Schur functions}, which are a generalization of the 
      (ordinary) quadratic Schur functions introduced by Darondeau--Pragacz, 
         and establish some Gysin formulas for the universal 
          quadratic Schur functions  as an application of our Gysin formulas. 
\end{abstract}

\maketitle


\section{Introduction}     \label{sec:Introduction}  
In  \cite{Darondeau-Pragacz2017},   Darondeau and Pragacz  established 
 push-forward 
(Gysin) formulas for all flag bundles  of types $A$, $B$, $C$, and $D$ in the 
framework of the Chow theory or  ordinary cohomology theory.  In this paper, 
we shall call their formulas {\it Darondeau--Pragacz formulas}.\footnote{
D--P formulas for short.  
}  
The main  idea of their  proof  is  to iterate the well-known push-forward formula 
for a projective bundle in types $A$ and $C$,  and for a quadric bundle in types $B$ and $D$.  
Thus,   the method  used in their paper is quite simple. 
Nevertheless,  the obtained formulas are quite useful and have many applications.  
For example, they used their formulas to obtain a new determinantal formula for 
Schur functions.  Moreover,  they introduced the {\it quadratic Schur functions}, 
which are the ``type $BCD$ analogues'' of the ordinary Schur functions, 
and established some Gysin formulas involving these functions.  
The purpose of the present paper is  two-fold:  First, we  generalize the D--P formulas 
in the ordinary cohomology theory to  the complex cobordism theory. 
Since the complex cobordism theory is {\it universal} 
among {\it complex-oriented} (in the sense of Adams 
\cite{Adams1974}) {\it generalized cohomology theories} 
(see Quillen \cite{Quillen1969}), our formulas  readily  yield 
the Gysin formulas for flag bundles in other complex-oriented 
generalized cohomology theories under the specialization of 
the universal formal group law.  For instance, one can obtain 
the Gysin formulas for flag bundles in the {\it complex $K$-theory}.  
Second,  we  generalize the quadratic Schur functions to the complex cobordism 
theory. 
When performing this generalization, we utilize  the Gysin formulas 
as a useful tool  to obtain the characterization of the quadratic Schur functions 
and their generalization.

To accomplish our  first  purpose, we need to formulate the following 
two items: 
\begin{itemize} 
\item  Segre classes of a complex vector bundle  in complex cobordism, 
\item  Gysin formula for a projective bundle in complex cobordism. 
\end{itemize} 
Recently Hudson--Matsumura \cite{Hudson-Matsumura2019} 
introduced the Segre classes  of a complex vector bundle 
in the {\it algebraic cobordism theory}, and their definition applies to the 
complex cobordism theory as well (see \S \ref{subsubsec:SegreClasses(ComplexCobordism)}).  
As for the latter,  Quillen \cite{Quillen1969} described 
the Gysin homomorphism of a projective bundle using the {\it residue 
symbol}
(see  \S \ref{subsubsec:FundamentalGysinFormulaProjectiveBundle(ComplexCobordism)}).  
Having prepared these two notions,  we can  establish 
the Gysin formulas for  full flag bundles   along the same lines 
as in \cite{Darondeau-Pragacz2017}.    To derive the formulas for 
general partial flag bundles from those of full flag bundles,  we utilize
the same idea as described in Damon \cite{Damon1973}.  In this process, 
we noticed  that  the {\it universal Schur functions} (corresponding to 
the empty partition) were needed. These functions are the complex cobordism 
analogues of the ordinary Schur polynomials, and were  first 
introduced by Fel'dman \cite{Fel'dman2003} (see also 
Nakagawa--Naruse \cite{Nakagawa-Naruse2016}). 
Using these functions, we can generalize  the D--P formulas in 
the ordinary cohomology theory  to the complex cobordism theory. 
Our main results are Theorems \ref{thm:TypeAD-PFormula(ComplexCobordism)}, 
\ref{thm:TypeCD-PFormula(ComplexCobordism)}, and \ref{thm:TypeBDD-PFormula(ComplexCobordism)}.   
Furthermore, it turned out that  most of the results 
established in \cite{Darondeau-Pragacz2017} can be generalized to 
the complex cobordism setting.    
In fact,   a   characterization of 
the  quadratic Schur functions  via  Gysin formulas  
immediately  leads to the definition of 
 the {\it  universal quadratic Schur functions}  
(see Definition \ref{defn:DefinitionUQSF}).  
  Then,  as applications of our Gysin formulas in comlex cobordism, 
we   give a certain Gysin formula (Proposition \ref{prop:GeneralizationPragacz-RatajskiFormula(ComplexCobordism)}), 
which is a complex cobordism analogue of the Pragacz--Ratajski formula \cite{Pragacz-Ratajski1997}. 
 The {\it generating function}  for  the universal quadratic Schur functions
is also obtained (Theorem \ref{thm:GFUQSF}),  
which immediately yields a determinantal  formula for 
the $K$-theoretic quadratic Schur functions under
the specialization (Theorem \ref{thm:DeterminantalFormulaKQSF}).\footnote{
For 
further applications of Gysin formulas in complex cobordism and techniques of 
generating functions,  readers are 
referred to a companion paper \cite{Nakagawa-Naruse(GFUHLPQF)}.  
}

\subsection{Organization of the paper}
The remainder of this paper is organized as follows:   Section  \ref{sec:NotationConventions} is 
a  preliminary section, that  gives a brief account of 
the complex cobordism theory,  universal formal group law, 
universal Schur functions, and flag bundles associated with vector bundles, 
  which will be used throughout this 
paper.  Section  \ref{sec:D-PFormulas(ComplexCobordism)} is the main body 
of the paper, and establishes the D--P formulas of types $A$, $B$, $C$, and 
$D$  in complex cobordism.  As mentioned above,  three items, namely, 
Segre classes in complex cobordism, Gysin formula for a projective bundle  
in complex cobordism, and universal Schur functions, play a significant role. 
The final section, Section \ref{sec:Applications}, deals with  
applications of our Gysin formulas, and discusses various properties of the 
universal quadratic Schur functions in detail.  In Appendix (Section 5), 
 Quillen's residue  formula will be computed explicitly.

\vspace{0.3cm}  

\textbf{Acknowledgments.} \quad    
We would like to thank Piotr Pragacz and Tomoo Matsumura  for 
 valuable discussions, and
Eric Marberg for pointing  out an ambiguity in our 
treatment of the residue symbol $\underset{t = 0}{\mathrm{Res}'}$ 
(see \S \ref{subsubsec:FundamentalGysinFormulaProjectiveBundle(ComplexCobordism)}
and \S \ref{subsec:QuillenResidueFormula}).   
Thanks are also due to the anonymous referee whose careful reading and useful comments 
greatly improved our previous manuscript.   
The  first author is partially supported by 
JSPS KAKENHI Grant Number  JP18K03303, Japan Society for  the Promotion of Science. 
The second author is partially supported by 
JSPS KAKENHI Grant Number  JP16H03921, Japan Society for  the Promotion of Science.

\section{Notation and conventions}  \label{sec:NotationConventions} 
\subsection{Complex cobordism theory}    \label{subsec:ComplexCobordismTheory}  
{\it Complex cobordism theory} $MU^{*}(-)$ is a generalized cohomology 
theory associated with the {\it Milnor--Thom spectrum} $MU$
(for a detailed account of the  complex cobordism theory,  readers are referred to 
e.g.,  Adams \cite{Adams1974}). 
According to Quillen \cite[Proposition 1.2]{Quillen1971}, 
for a manifold $X$, $MU^{q}(X)$ can be identified with 
the set of  {\it cobordism classes}  of {\it proper}, {\it complex-oriented} 
maps of dimension $-q$.  Thus,  a map of manifolds $f: Z  \longrightarrow X$, which is 
complex-oriented in the sense of Quillen, determines a class 
denoted by $[Z \overset{f}{\rightarrow} X]$, or   simply $[Z]$ 
in $MU^{q}(X)$.   The coefficient ring of this theory is given by 
$MU^{*} := MU^{*}(\mathrm{pt})$, where $\mathrm{pt}$ is a space 
consisting of a single point.    With this geometric interpretation of 
$MU^{*}(X)$,  the {\it Gysin map} can be defined as follows:  
For a proper complex-oriented map $g: X \longrightarrow Y
$ of dimension $d$, the Gysin map 
\begin{equation*} 
    g_{*}:  MU^{q}(X)  \longrightarrow  MU^{q - d}(Y)
\end{equation*} 
is defined by sending the cobordism class $[Z \overset{f}{\rightarrow} X]$ 
into the class $[Z  \overset{g \circ f}{\rightarrow}  Y]$.

Complex cobordism theory $MU^{*}(-)$ is equipped with the ``generalized'' 
Chern  classes. 
More precisely, for a rank $n$ complex vector bundle $E$  over a space 
$X$, one can define the {\it $MU^{*}$-theory Chern classes}  
$c^{MU}_{i}(E)  \in MU^{2i}(X)$ for $i  = 0, 1,  \ldots, n$, 
which have the usual properties of the ordinary 
Chern classes in cohomology (see Conner-Floyd \cite[Theorem 7.6]{Conner-Floyd1966}, 
Switzer \cite[Theorem 16.2]{Switzer1975}).

Let $\C P^{\infty}$ be an  infinite complex projective space, and 
$\eta_{\infty}  \longrightarrow \C P^{\infty}$ the {\it Hopf line bundle} 
on $\C P^{\infty}$.     Note that  $\C P^{\infty}$ is homotopy equivalent 
to the classifying space $BU(1)$ of the unitary group $U(1)$.  
Let $x = x^{MU}$ be the $MU^{*}$-theory first Chern class of the line bundle 
$\eta_{\infty}^{\vee}$, dual of $\eta_{\infty}$.  Then,   it is well-known that 
$MU^{*}(\C P^{\infty})  \cong MU^{*} [[x ]]$.  
 Denote    the natural projection onto the $i$-th factor ($i = 1, 2$) 
by $\pi_{i}: \C P^{\infty} \times \C P^{\infty} 
\longrightarrow \C P^{\infty}$.   Then,   the {\it product map} $\mu:  \C P^{\infty} 
\times \C P^{\infty}  \longrightarrow \C P^{\infty}$ is defined as the 
classifying map  of the line bundle $\pi_{1}^{*}\eta_{\infty} \otimes \pi_{2}^{*} \eta_{\infty}$ 
over $\C P^{\infty} \times \C P^{\infty}$.  
Applying the functor $MU^{*}(-)$ to the  map 
$\mu$,   we obtain 
\begin{equation*} 
    \mu^{*}:   MU^{*}(\C P^{\infty}) \cong   MU^{*}[[x ]]  
          \longrightarrow   MU^{*}(\C P^{\infty} \times \C P^{\infty}) 
 \cong MU^{*} [[x_{1}, x_{2}]], 
\end{equation*} 
where $x_{i}  = x_{i}^{MU}$ is the $MU^{*}$-theory first Chern class 
of the line bundle $\pi_{i}^{*} \eta_{\infty}^{\vee} \; (i = 1, 2)$.    
From this, one obtains the formal power series in two variables: 
\begin{equation*} 
    \mu^{*}(x)  =   F^{MU}(x_{1}, x_{2})  = \sum_{i, j \geq 0}  a_{i, j}x_{1}^{i} x_{2}^{j} 
\quad (a_{i, j}  = a_{i, j}^{MU}   \in MU^{2(1 - i - j)}). 
\end{equation*} 
Therefore,  for    complex line bundles $L$ and  $M$  over the same 
base space,   the following formula holds:  
\begin{equation*} 
     c_{1}^{MU}(L  \otimes M)  =  F^{MU} (c_{1}^{MU}(L),  c_{1}^{MU}(M)).  
\end{equation*} 
Quillen \cite[\S 2]{Quillen1969} 
showed that the formal power series $F^{MU}(x_{1}, x_{2})$ 
is a formal group law  over $MU^{*}$.  
Moreover, he also showed that the formal group law $F^{MU} (x_{1}, x_{2})$ over $MU^{*}$ 
is  {\it universal}  in the sense that given any formal group law $F(x_{1}, x_{2})$ 
over a commutative ring $R$ with unit, there exists a unique ring homomorphism 
$\theta:   MU^{*}  \longrightarrow R$  carrying $F^{MU}$ to $F$. 
Topologically,  Quillen's result claims that the 
 complex cobordism theory $MU^{*}(-)$ is {\it universal} among  complex-oriented 
generalized cohomology theories.   It has been  known since Quillen that a  complex-oriented 
generalized cohomology theory gives rise to a formal group law in  the same manner 
as that above.  
For instance, the ordinary cohomology theory (with integer coefficients) $H^{*}(-)$ 
corresponds to the 
additive formal group law $F^{H}(x_{1}, x_{2}) = x_{1} + x_{2}$, 
and the (topological) complex $K$-theory   $K^{*}(-)$  corresponds to the multiplicative 
formal group law  $F^{K}(x_{1}, x_{2}) = x_{1} + x_{2} - \beta x_{1}x_{2}$
for some unit  $\beta \in K^{-2} = K^{-2}(\mathrm{pt})  =  \tilde{K}^{0}(S^{2})$. 
Here,  a comment concerning the $K$-theoretic Chern classes is in order: 
 Following Levine--Morel \cite[Example 1.1.5]{Levine-Morel2007},   
   for a complex line bundle $L \longrightarrow X$,  we define the $K$-theoretic first Chern class 
$c^{K}_{1}(L)$ to be $\beta^{-1} (1 - L^{\vee})  \in K^{2}(X)$.   Then, the 
corresponding formal group law is given as stated above.

\subsection{Lazard ring $\L$ and  universal formal group law $F_{\L}$}    \label{subsec:UFGL}   
Quillen's result  in the previous subsection 
implies that the formal group law $F^{MU}$ over $MU^{*}$ is
 identified with 
the so-called  {\it Lazard's universal formal group law}  $F_{\L}(u, v)$ 
  over the {\it Lazard ring  $\L$}, 
 which we briefly recall from   Levine--Morel's book \cite[Chapters 1 and 2]{Levine-Morel2007}. 
Accordingly,  the additive formal group law is denoted by $F_{a} (u, v) = u + v$ in 
place of $F^{H}$, 
and the multiplicative formal group law by $F_{m}(u, v) = u + v - \beta uv$  
in place of $F^{K}$  in the following.    
In \cite{Lazard1955},  Lazard constructed  the  universal 
formal group law
\begin{equation*} 
F_{\L} (u,v) =   u + v + \sum_{i,j \geq 1} a^{\L}_{i,j} u^{i}  v^{j}   \in    \L[[u,v]] 
\end{equation*} 
over  the ring $\L$, where  $\L$ is   the {\it Lazard ring}, and he showed that 
it  is isomorphic to the polynomial ring 
in a countably infinite number of variables with  integer coefficients.   
$F_{\L} (u, v)$  
is a formal power series in $u$,  $v$ with coefficients $a^{\L}_{i,j}  \in \L$
that  satisfies the axioms of the  formal group law: 
\begin{enumerate} 
\item [(i)]  $F_{\L} (u, 0) = u$, $F_{\L} (0, v) = v$, 
\item [(ii)] $F_{\L}(u, v) = F_{\L}  (v, u)$, 
\item [(iii)]  $F_{\L}(u, F_{\L} (v, w)) =  F_{\L} (F_{\L}(u, v), w)$. 
\end{enumerate} 
For the universal formal group law, we shall use the notation
\begin{equation*} 
\begin{array}{llll} 
  &  u  +_{\L}    v =  F_{\L}(u,  v)   \quad  & \text{(formal sum)}, \medskip \\ 
  &  \overline{u} =   [-1]_{\L} (u)  = \chi_{_{\L}}(u)  & \text{(formal  inverse of} \;   u),  \medskip  \\
  &   u -_{\L} v =  u +_{\L}  [-1]_{\L}(v) = u +_{\L}  \overline{v}   &  \text{(formal subtraction)}. 
\medskip  
\end{array}
\end{equation*}
Furthermore,     we define $[0]_{\L}(u) := 0$, and inductively,  
   $[n]_{\L}(u)  :=  [n-1]_{\L}(u) +_{\L} u$
for a positive integer  $n \geq 1$.  We also define   
 $[-n]_{\L}(u) := [n]_{\L}([-1]_{\L}(u))$ for $n \geq 1$.     
We call $[n]_{\L}(u)$ the {\it $n$-series} in the sequel
(we often drop $\L$ from the notation, and simply write $[n](u)$ for simplicity). 
Denote  the  {\it logarithm}  (see \cite[Lemma 4.1.29]{Levine-Morel2007})   
of  $F_{\L}$    by  $\ell_{\L}  (u)   \in \L \otimes \Q [[u]]$, i.e.,  
a unique formal power series with leading term $u$ 
such that   
\begin{equation*} 
\ell_{\L}(u +_{\L} v)  =   \ell_{\L} (u) + \ell_{\L} (v).   
\end{equation*} 
The Lazard ring  $\L$ can be    graded  by assigning   each coefficient 
$a^{\L}_{i, j}$  the degree  
$1 - i - j  \; (i, j \geq 1)$.  
This grading   makes $\L$ into the graded ring over the integers $\Z$. 
Be aware that in topology, it is customary to give   $a^{\L}_{i, j}$ the {\it cohomological} 
degree $2(1 - i - j)$.

For the complex $K$-theory, we shall use the following notation: 
\begin{equation*} 
\begin{array}{llll}  
     &  u \oplus v  =  F_{m}(u, v)  = u + v - \beta uv, \medskip \\
     &  \overline{u}  =  \dfrac{-u}{1 - \beta u}, \medskip \\
     &  u \ominus v  = u  \oplus  \overline{v}  =  \dfrac{u - v}{1 - \beta v}.  \medskip 
\end{array}  
\end{equation*}

\subsection{Universal Schur functions}    \label{subsec:USF}  
Throughout this paper, we use the notation  concerning partitions 
as in Macdonald's book \cite[Chapter I]{Macdonald1995}.   
A partition $\lambda$ is a non-increasing sequence $\lambda = (\lambda_{1}, \lambda_{2}, 
\ldots, \lambda_{k})$ of non-negative integers  such that $\lambda_{1} \geq \lambda_{2} \geq \cdots 
\geq \lambda_{k} \geq 0$.   As is customary, we do not distinguish between two 
such sequences that  differ only by a finite or infinite sequence of zeros at the end.  
The non-zero $\lambda_{i}$'s   are called  parts of $\lambda$, and the number 
of parts is the length of $\lambda$, denoted by $\ell (\lambda)$.  
The sum of the parts is the weight of $\lambda$, denoted by $|\lambda|$.  
If $|\lambda| = n$,   then we say  that $\lambda$ is a partition of $n$.  
If $\lambda$ and  $\mu$ are partitions,  then we  write $\lambda \subset \mu$ to 
mean that $\lambda_{i}  \leq \mu_{i}$ for all $i \geq 1$.   
In what follows, the set of all partitions  of length $\leq n$ is denoted 
by   $\mathcal{P}_{n}$. 
For a non-negative  integer $n$,   we set $\rho_{n} :=(n, n-1, \ldots, 2, 1)$ 
($\rho_{0}$ as understood to be the unique partition of $0$, 
which we denote by just  $0$ or $\emptyset$).  
The partition $(\underbrace{k, k, \ldots, k}_{l})$ is abbreviated as  $(k^{l})$, 
or just $k^{l}$. 
Let $\lambda, \mu \in \mathcal{P}_{n}$ be partitions. Then,   $\lambda + \mu$ 
is a partition  
defined by $(\lambda + \mu)_{i} := \lambda_{i} + \mu_{i} \; (i = 1, 2, \ldots, n)$. 
Given a partition $\lambda \in \mathcal{P}_{n}$ and a positive integer $k$, 
we denote the partition $(k\lambda_{1}, \ldots, k \lambda_{n})$ by $k \lambda$.

Now,  let us recall the definition of the {\it universal Schur functions} 
(see Nakagawa--Naruse \cite[Definition 4.10]{Nakagawa-Naruse2016}):  
   Let $\x_{n} = (x_{1}, x_{2}, \ldots, x_{n})$ be  $n$ independent  variables.  
For a partition $\lambda \in \mathcal{P}_{n}$,   
 we use the notation  
   $\bm{x}^{\lambda}  :=  \prod_{i=1}^{n}  x_{i}^{\lambda_{i}}$
($x_{i}^{0} := 1$).   Then,   the {\it universal Schur function}  $s^{\L}_{\lambda}(\x_{n})$ in the variables 
$\x_{n} = (x_{1}, x_{2}, \ldots, x_{n})$ corresponding to the partition 
$\lambda  \in \mathcal{P}_{n}$   is defined as 
\begin{equation}     \label{eqn:DefinitionUSF} 
    s^{\L}_{\lambda}(\x_{n}) :=   \sum_{w \in S_{n}}  w \cdot 
        \left [     \dfrac{\bm{x}^{\lambda + \rho_{n-1}}}{ \prod_{1 \leq i < j \leq n}  (x_{i} +_{\L} \overline{x}_{j})} 
         \right ], 
\end{equation} 
where the symmetric group $S_{n}$ acts on the variables 
$\x_{n} = (x_{1}, \ldots, x_{n})$ by permutations.  This  is a generalization of the 
usual Schur polynomial $s_{\lambda}(\x_{n})$  (see e.g., Macdonald \cite[Chapter I, (3.1)]{Macdonald1995}), 
and the {\it Grassmann  Grothendieck polynomial} 
$G_{\lambda}(\x_{n})$  (Buch \cite[\S 2]{Buch2002(Acta)}).   In fact, 
under the specialization  from  the universal formal group law  $F_{\L}(u, v) = u +_{\L} v$
 to the additive one, $F_{a}(u, v) =  u + v$ (resp. the multiplicative one, 
$F_{m}(u, v) =  u \oplus v$),    the function $s^{\L}_{\lambda}(\x_{n})$ is reduced to 
$s_{\lambda}(\x_{n})$ (resp.  $G_{\lambda}(\x_{n})$).  
Note that $s_{\lambda}(\x_{n})$ is a polynomial in $\x_{n}$ with integer coefficients, and 
$G_{\lambda}(\x_{n})$ is also a polynomial in $\x_{n}$ with coefficients in $\Z[\beta]$, 
whereas $s^{\L}_{\lambda}(\x_{n})$ is a formal power series in $\x_{n}$ with 
coefficients in $\L$.  Moreover, unlike the Schur and Grothendieck polynomials, 
the function $s^{\L}_{\emptyset} (\x_{n})$ corresponding to the empty partition 
$\emptyset = (0^{n})$ is not equal to $1$.  For example, we have 
\begin{equation*} 
   s^{\L}_{\emptyset} (\x_{2})  =  \dfrac{x_{1}}{x_{1} +_{\L}  \overline{x}_{2}} 
                                        +  \dfrac{x_{2}}{x_{2} +_{\L}  \overline{x}_{1}} 
                 =  1 + a_{1, 2}^{\L} x_{1} x_{2} +  \cdots   \neq 1.  
\end{equation*} 

In the later section (\S \ref{sec:Applications}), we need to extend 
the above definition (\ref{eqn:DefinitionUSF}) to arbitrary sequences of non-negative integers. 
For such a sequence $I = (I_{1}, \ldots, I_{n})  \in (\Z_{\geq 0})^{n}$, 
the corresponding universal Schur function $s^{\L}_{I}(\x_{n})$ is defined 
by the same expression as (\ref{eqn:DefinitionUSF}).

\subsection{Flag bundles associated with vector bundles}   \label{subsec:FlagBundles}  
Darondeau--Pragacz formulas describe  Gysin maps for flag bundles associated with 
complex vector bundles.   
In this subsection, we prepare the necessary notations  and terminologies 
concerning flag bundles used in this paper (for more details, readers are referred to 
Darondeau--Pragacz \cite[Sections  1, 2, and  3]{Darondeau-Pragacz2017},  
Edidin--Graham \cite[Section  6]{Edidin-Graham1995}, 
Fulton--Pragacz \cite[Section 6.1]{Fulton-Pragacz1998}). 
Let $E \longrightarrow X$ be a rank $N$ complex vector bundle of {\it type} $Y$, where 
$Y$ stands for $A$, $B$, $C$, or $D$.  The situations considered here are as follows: 
\begin{itemize} 
\item  Type $A$: $N = n$, and no conditions on $E$;

\item Type $C$: $N = 2n$, and $E$ is equipped with a non-degenerate {\it symplectic} form $\langle -, - \rangle$ 
        with values in a certain line bundle $L$;  

\item Type $BD$: $N = 2n + 1$ for type $B$, or $N = 2n$ for type $D$, and 
          $E$ is equipped with a non-degenerate {\it orthogonal} form $\langle -, - \rangle$ 
        with values in a certain line bundle $L$.  
\end{itemize}  
Let $0  <   q_{1} < q_{2} < \cdots < q_{m} \leq n$ be a sequence of integers. 
We denote by  $\varpi_{q_{1}, \ldots, q_{m}} =   \varpi_{q_{1}, q_{2}, \ldots, q_{m}}^{Y}:  
\F \ell^{Y}_{q_{1}, q_{2}, \ldots, q_{m}} (E)   \longrightarrow X$ the corresponding 
partial flag bundle of type $Y$. 
On $\F \ell^{Y}_{q_{1}, \ldots, q_{m}} (E)$,  there 
exists a universal  (isotropic for $Y = B, C$ or $D$)   flag of subbundles of $\varpi_{q_{1}, \ldots, q_{m}}^{*}(E) = E$, 
\begin{equation*}   
    0  \subset U_{q_{1}} \subset U_{q_{2}} \subset \cdots \subset 
      U_{q_m}  \subset E,  
\end{equation*} 
where $\rank \, U_{q_{i}} = q_{i}  \; (i = 1, \ldots, m)$ (throughout the paper, 
the subscripts of the bundles will 
denote the rank unless otherwise specified).  
 As a special case when $q_{k} = k \; (k = 1, \ldots, m)$, we  obtain the {\it full}   flag 
bundle  $\varpi_{1, 2, \ldots, m} = \varpi_{1, 2, \ldots, m}^{Y}: \F \ell^{Y}_{1, 2, \ldots, m}(E)  \longrightarrow
X$  of type $Y$. 
In particular,  the full flag bundle 
$\varpi^{Y} =  \varpi_{1, 2, \ldots, n}^{Y}:   \F \ell^{Y} (E) =   \F \ell^{Y}_{1, 2, \ldots, n}(E) \longrightarrow X$ 
is  the   {\it complete} flag bundle of type $Y$.\footnote{
We adopted the terminology used  in Darondeau--Pragacz \cite[\S 1.2]{Darondeau-Pragacz2017}. 
}    On $\F \ell^{Y}(E)$, we have the universal  (isotropic for $Y = B, C$, or $D$) flag of subbundles 
\begin{equation}   \label{eqn:UniversalFlagSubbundles}
     0 = U_{0}  \subset U_{1}  \subset \cdots \subset U_{n} \subset E,     
\end{equation} 
and we put 
\begin{equation*} 
        y_{i} :=   c_{1}^{MU}  ((U_{i}/U_{i-1})^{\vee})   \in MU^{2}(\F \ell^{Y} (E))  
 \quad (i  = 1, 2, \ldots,  n). 
\end{equation*} 
These  are the $MU^{*}$-theory Chern roots of $U_{n}^{\vee}$.                   
For $Y = B, C$, or $D$, the universal isotropic flag of subbundles (\ref{eqn:UniversalFlagSubbundles})
 can be extended to the ``ordinary'' universal flag of subbundles of $E$, 
\begin{equation*} 
     0 = U_{0}  \subset U_{1}  \subset \cdots \subset U_{n-1}  \subset U_{n}  \subset  U_{n}^{\perp}  
     \subset  U_{n-1}^{\perp} \subset  \cdots  \subset U_{2}^{\perp}  \subset U_{1}^{\perp}  \subset E.    
\end{equation*} 
Here $U_{i}^{\perp}$ denotes its {\it complement} with respect to the 
given symplectic or orthogonal form $\langle -,  - \rangle$.  
Then,  using an  isomorphism of vector bundles 
              $U_{i-1}^{\perp}/U_{i}^{\perp}  
          \overset{\sim}{\longrightarrow}   (U_{i}/U_{i-1})^{\vee}  \otimes L$
    induced from the given  symplectic or orthogonal form, 
  the whole of Chern roots of $E^{\vee}$ are given as follows: 
\begin{itemize} 
  \item    Type $C$ case:   $y_{1}, \ldots, y_{n},       \overline{y}_{1} +_{\L}   \overline{z}, \ldots, \overline{y}_{n} +_{\L}  \overline{z}$,

 \item   Type  $B$ case:   
         $y_{1}, \ldots, y_{n}, y_{n + 1},  \overline{y}_{1} +_{\L} \overline{z}, \ldots, \overline{y}_{n} +_{\L} \overline{z}$,

\item  Type  $D$  case:   
        $y_{1}, \ldots, y_{n},   \overline{y}_{1} +_{\L} \overline{z}, \ldots, \overline{y}_{n} +_{\L} \overline{z}$,   
\end{itemize} 
where we set   $y_{n + 1} := c^{MU}_{1}((U_{n}^{\perp}/U_{n})^{\vee})$  and   $z := c^{MU}_{1}(L)$.

For $Y = A$ or $C$, it is well-known  that the full flag bundle $\F \ell^{Y}_{1, 2, \ldots, q}(E)$  is constructed as 
a chain of projective bundles, 
and we apply the push-forward formula for   a projective bundle (see (\ref{eqn:FundamentalFormula(TypeA)}))  
repeatedly to this chain to obtain our Gysin formulas.  
For $Y = B$ or $D$,  the full flag bundle $\F \ell^{Y}_{1, \ldots, q}(E)$ is constructed as a chain of   
quadric bundles of  isotropic lines,  and    we apply the push-forward formula for   a quadric  bundle (see (\ref{eqn:FundamentalFormula(TypeBD)}))  
repeatedly to this chain to obtain our Gysin formulas.  
The quadric bundle of isotropic lines associated with 
an orthogonal bundle $E  \longrightarrow X$ is  denoted by 
$\rho_{1}:    Q(E)  \longrightarrow X$ in the sequel. 
Let $\iota: Q(E)  \hooklongrightarrow P(E)$  be the natural inclusion.  
 Then,  we have  $\rho_{1}  =   \varpi_{1} \circ \iota$, 
where $\varpi_{1}:  P(E)  \longrightarrow X$ is the associated projective bundle 
of lines.

\section{Darondeau--Pragacz formulas in complex cobordism}  \label{sec:D-PFormulas(ComplexCobordism)}   
In this section, we  generalize the push-forward (Gysin)  formulas for flag bundles
due to Darondeau--Pragacz \cite{Darondeau-Pragacz2017} to the complex cobordism 
theory.    Their formulas will be referred to as the Darondeau--Pragacz formulas
in the sequel.       
The original Darondeau--Pragacz formulas 
were formulated in the Chow theory, or the ordinary cohomology theory. 
Their formulas are obtained by iterating the classical push-forward formula for a projective 
bundle associated with a complex vector bundle, and 
are  expressed by using the Segre classes,  
or  the Segre polynomials  of complex 
vector bundles.   Therefore,  to generalize their formulas to the complex 
cobordism theory,  we need to generalize   the notion of  the Segre classes
and the classical push-forward formula for a 
projective bundle  in  cohomology   to  complex cobordism.     These 
generalizations were essentially established by Hudson--Matsumura 
\cite{Hudson-Matsumura2019} 
and Quillen \cite{Quillen1969} respectively, and are recalled in the next subsection.

\subsection{From cohomology to complex cobordism} 
\subsubsection{Segre classes in complex cobordism}    \label{subsubsec:SegreClasses(ComplexCobordism)}  
 To formulate  the Darondeau--Pragacz formulas in complex cobordism,  
  we need a notion of  the Segre classes  of 
complex vector bundles in complex cobordism. 
 Fortunately such  a notion was  recently  introduced by Hudson--Matsumura 
\cite[Definition 4.3]{Hudson-Matsumura2019}. 
 More precisely, they defined  the Segre classes $\mathscr{S}_{m}(E) \; (m \in \Z)$ 
 in the algebraic cobordism theory $\Omega^{*}(-)$ 
  of  a complex vector bundle $E$    using the push-forward image 
of the projective bundle $G^{1}(E) \cong P(E^{\vee})$.  Notice that 
this definition is the  exact analogue  of the Segre classes in ordinary cohomology 
  given by Fulton \cite[\S 3.1]{Fulton1998}.      
$K$-theoretic generalization of Segre classes $G_{m}(E) \; (m \in \Z)$ 
 is also introduced in the same manner 
(Buch \cite[Lemma 7.1]{Buch2002(Duke)}).     
 The Segre classes  of  $E$ in the   complex 
cobordism theory $MU^{*}(-)$, denoted by   $\mathscr{S}^{\L}_{m}(E) \; (m \in   \Z)$, 
can be defined in exactly the same manner as above. 
Denote 
the generating function of the Segre classes, which we call the {\it Segre series}  
in complex cobordism, 
by
\begin{equation*} 
   \mathscr{S}^{\L} (E; u)  :=  \sum_{m \in \Z} \mathscr{S}^{\L}_{m}(E)  u^{m}.  
\end{equation*}   
More explicitly, the following expression has been obtained 
(see \cite[Theorem 4.6]{Hudson-Matsumura2019}):    
Let   $x_{1}, \ldots, x_{n}$ be  the $MU^{*}$-theory Chern roots of $E$.  
Then,  the Segre series $\mathscr{S}^{\L}(E; u)$  of $E$ is given by 
\begin{equation}   \label{eqn:SegreSeries(ComplexCobordism)}  
     \mathscr{S}^{\L}(E; u)  
        = \left.    \dfrac{1}{\mathscr{P}^{\L} (z)}   
                             \prod_{j=1}^{n}  \dfrac{z}{ z +_{\L}  \overline{x}_{j}}   
                                       \right |_{z = u^{-1}} 
                 =   \left.    \dfrac{1}{\mathscr{P}^{\L} (z)}   
                                   \dfrac{z^{n}}{\prod_{j=1}^{n} (z +_{\L}  \overline{x}_{j}) }   
                                       \right |_{z = u^{-1}} ,    
\end{equation}
where  $\mathscr{P}^{\L}(z) :=  1 + \sum_{i=1}^{\infty} a^{\L}_{i, 1} z^{i}$.

These Segre classes $\mathscr{S}^{\L}_{m}(E) \; (m \in \Z)$ are  natural 
generalizations of the (ordinary) Segre classes in cohomology.   
If we specialize the universal formal group law $F_{\L} (u, v) = u +_{\L} v$ to 
$F_{a}(u, v) = u + v$, then $\mathscr{S}^{\L}(E; u)$ reduces 
to 
\begin{equation*} 
        \left.    \prod_{j=1}^{n}   \dfrac{z}{z - x_{j}} \right |_{z = u^{-1}}   
   =   \prod_{j=1}^{n}  \dfrac{1}{1 - x_{j} u}   = \dfrac{1}{c(E; -u)}   
   =  \dfrac{1}{c(E^{\vee}; u)}  = s(E; u).  
\end{equation*} 
Here,  $c (E; u) := \sum_{i = 0}^{n} c_{i}(E) u^{i}$ (resp.  $s(E; u)  = \sum_{i \geq 0} s_{i}(E) u^{i}$) 
 denotes  the (ordinary)  Chern polynomial (resp.  Segre series)  
 of $E$.    
  This formula also implies 
that the $i$-th Segre class $s_{i}(E)$ is identified with the $i$-th 
complete symmetric polynomial $h_{i}(\x_{n})$, 
or  the Schur polynomial $s_{(i)}(\x_{n})$ corresponding to the one-row $(i)$
 in the variables  $\x_{n} = (x_{1}, \ldots, x_{n})$. 
The classes $\mathscr{S}^{\L}_{m}(E)$ 
 are also  generalizations of the $K$-theoretic Segre classes 
$G_{m}(E) \; (m \in \Z)$.   In fact, 
  if we specialize $F_{\L}(u, v) = u +_{\L} v$ to  $F_{m}(u, v) = u \oplus v$, then 
$\mathscr{S}^{\L} (E; u)$ reduces to 
\begin{equation}    \label{eqn:K-theoreticSegreSeries}  
  \left.    \dfrac{1}{1 - \beta z}   \prod_{j=1}^{n}  \dfrac{z}{z \ominus x_{j}} \right |_{z = u^{-1}}   
   =  \dfrac{1}{1 - \beta u^{-1}} \prod_{j=1}^{n}  \dfrac{1 - \beta x_{j}}  
                                                                        { 1 - x_{j} u} 
   =  \dfrac{1}{1 - \beta u^{-1}}   \dfrac{c^{K} (E; -\beta)}
                                                     {c^{K}  (E; -u)}, 
\end{equation} 
which is the $K$-theoretic Segre series 
$G(E; u) = \sum_{m \in \Z}  G_{m} (E) u^{m}$  of $E$ 
given by 
Hudson--Ikeda--Matsumura--Naruse \cite[Theorem 2.8]{HIMN2017}.    
Here, $c^{K}(E; u) =  \sum_{i = 0}^{n} c^{K}_{i}(E) u^{i}$ is the 
$K$-theoretic Chern polynomial of $E$.   For $m \geq 1$, the $m$-th
$K$-theoretic Segre class $G_{m}(E)$ is identified with the $m$-th 
 Grothendieck polynomial $G_{m}(\x_{n})$ corresponding to the 
one-row $(m)$.    We remark that 
our Segre class $\mathscr{S}^{\L}_{m}(E)  \; (m \geq 1)$
 in complex cobordism 
can be identified with the {\it new universal Schur function} 
$\mathbb{S}^{\L}_{m} (\x_{n})$ 
(see  Nakagawa--Naruse 
\cite[Remark 5.10]{Nakagawa-Naruse2018}).

\begin{rem}   \label{rem:GeometricMeaningP^L(z)} 
The formal power series $\mathscr{P}^{\L} (z)  \in \L [[z]]$ has the following geometric 
meaning$:$   By the argument in Quillen \cite[\S 1]{Quillen1969}, 
we have 
$\dfrac{\partial F_{\L}}{\partial v}(z,  0)  =   \mathscr{P}^{\L} (z)$, 
and hence $\ell_{\L}' (z)  = \dfrac{1}{\mathscr{P}^{\L} (z)}$.   
As a result of Mi\v{s}\v{c}enko $($see Adams {\rm \cite[Chapter II, Corollary 9.2]{Adams1974}}$)$,  
we know that 
\begin{equation*} 
     \ell_{\L} (z)  =  \sum_{m =0}^{\infty}  \dfrac{[\C P^{m}]}{m + 1} z^{m+1}, 
\end{equation*} 
 where  $[ \C P^{m}]  \in MU^{-2m}  = \L^{-2m}$   
is  the cobordism class of $\C P^{m}$.   Therefore,   we have 
\begin{equation*} 
       \dfrac{1}{\mathscr{P}^{\L}(z)}    = \sum_{m = 0}^{\infty} [\C P^{m}]  z^{m}.  
\end{equation*} 
\end{rem}  

\subsubsection{Fundamental Gysin formula for   projective bundle in complex cobordism}  
\label{subsubsec:FundamentalGysinFormulaProjectiveBundle(ComplexCobordism)}  
Let $E \longrightarrow X$ be a complex vector bundle of rank $n$,  and 
$\varpi_{1}:  P(E)  \longrightarrow X$ the associated projective 
bundle of lines in $E$.  Denote the tautological line bundle on $P(E)$ by $U_{1}$.  
Put $y_{1} := c_{1}^{MU}(U_{1}^{\vee})  \in MU^{2}(P(E))$.  
In \cite{Quillen1969},  Quillen described the Gysin map 
$\varpi_{1 *}:  MU^{*}(P(E))  
\longrightarrow MU^{*}(X)$.   In our notation, his formula is 
stated as follows:  
\begin{theorem} [Quillen \cite{Quillen1969}, Theorem 1]  
For a polynomial $f(t)   \in MU^{*}(X)[t]$,  
  the Gysin map $\varpi_{1 *}:  MU^{*}(P(E))  
\longrightarrow MU^{*}(X)$ is given by the {\it residue formula}   
\begin{equation}     \label{eqn:QuillenResidueFormula}    
    \varpi_{1 *}  (f(y_{1}))    =   \underset{t  = 0}{\Res'} \,  
                                                 \dfrac{ f(t) 
                                                          }  
                                             { \mathscr{P}^{\L}(t)  
                                                \prod_{i=1}^{n}  (t  +_{\L}  \overline{y}_{i})  
                                            }, 
\end{equation}  
where $y_{1},  \ldots, y_{n}$ denote the $MU^{*}$-theory Chern roots 
of $E^{\vee}$.  
\end{theorem} 
 Here,   the symbol $\underset{t  = 0}{\Res'} \,   F(t)$ is understood to be 
the coefficient of 
$t^{-1}$ in the {\it formal} Laurent series $F(t)$.  
However, we  must be careful when 
we apply the operation $\underset{t=0}{\Res'}$ to the formal Laurent series. 
For example, let us  consider the rational function $f(t) = 1/(1 - t)$.  
Then, on the one hand,  we have 
 $f(t) = 1 + t + t^2 + \cdots$ when expanded as a formal power 
series in $t$, and therefore $\underset{t = 0}{\Res'}f(t) = 0$.  
On the other hand, one can expand $f(t)$ as a formal power series in $t^{-1}$
so that $f(t) =  -t^{-1} - t^{-2} - \cdots$.  Therefore,  we have  $\underset{t = 0}{\Res'}f(t) 
= -1$.   Thus, we must specify  how to expand $f(t)$ as a formal power series in $t$ or $t^{-1}$ 
when we apply $\underset{t=0}{\Res'}$ to the rational function $f(t)$. 
In the above formula (\ref{eqn:QuillenResidueFormula}),  we expand the 
rational function of  the right-hand side in accordance with the 
following convention (for this interpretation of Quillen's result, see also  Naruse \cite[Lemma 4]{Naruse2018}): 
As mentioned in Remark \ref{rem:GeometricMeaningP^L(z)},  
we always treat $1/\mathscr{P}^{\L}(t)$ as a formal power series in $t$, 
that is, $1/\mathscr{P}^{\L}(t)  =  \sum_{m=0}^{\infty}  [\C P^{m}]  t^{m}$.  
For the product $1/\prod_{i=1}^{n}(t +_{\L} \overline{y}_{i})$,   
we expand this as a formal power series in $t^{-1}$ by using the 
following expansion: 
\begin{equation*} 
       \dfrac{1}{t +_{\L}  \overline{y}_{i}}    =   \dfrac{\mathscr{P}^{\L}  (t, y_{i}) }{t - y_{i}}  
   =  t^{-1}  \times \mathscr{P}^{\L} (t, y_{i})  \times \dfrac{1}{1 - y_{i}t^{-1}} 
   =  t^{-1}  \times \mathscr{P}^{\L}(t, y_{i})  \times 
        \sum_{k=0}^{\infty}  y_{i}^{k}  t^{-k},   
\end{equation*} 
where  $\mathscr{P}^{\L}(t, y_{i}) :=  \dfrac{t - y_{i}}{t +_{\L} \overline{y}_{i}}$.  
For further calculations, readers are referred to Appendix  \ref{subsec:QuillenResidueFormula}.

Following Darondeau--Pragacz \cite[p.2, (2)]{Darondeau-Pragacz2017}, 
we reformulate  the above formula in a more convenient form.   To do so, 
we use the same notation as  in \cite{Darondeau-Pragacz2017}.  
For a monomial $m$ of a Laurent polynomial $F$, we denote the coefficient of $m$ in $F$ by $[m](F)$.  
With these conventions and    the Segre series (\ref{eqn:SegreSeries(ComplexCobordism)}),  
the residue formula (\ref{eqn:QuillenResidueFormula}) becomes 
\begin{equation}    \label{eqn:FundamentalFormula(TypeA)}  
    \varpi_{1 *} (f(y_{1}))  
                      = [t^{-1}]  \left (  f(t)  \cdot    t^{-n}  \mathscr{S}^{\L} (E^{\vee}; 1/t) 
                                                  \right )  \medskip  \\ 
                               =   [t^{n-1}]  (f(t)  \mathscr{S}^{\L}(E^{\vee}; 1/t)).    
\end{equation}  
This is the fundamental formula for establishing more general Gysin 
formulas for   general flag bundles.

\subsection{Darondeau--Pragacz formula of type $A$ in complex cobordism}

With the above preliminaries, we can extend  
Darondeau--Pragacz formula \cite[Theorem 1.1]{Darondeau-Pragacz2017}
in the following manner:  Let $E  \longrightarrow X$ be a complex vector bundle 
of rank $n$.   Given a sequence of integers $q_{0} = 0 < q_{1} < \cdots < q_{m}  \leq  n  = q_{m + 1}$,   
we set $q := q_{m}$.  
Then,  the following Gysin formula holds for the partial flag bundle 
$\varpi_{q_{1}, \ldots, q_{m-1}, q} =\varpi^{A}_{q_{1}, \ldots, q_{m-1}, q} :  \F \ell_{q_{1}, \ldots, q_{m-1}, q}(E)  
  =  \F \ell^{A}_{q_{1}, \ldots, q_{m-1}, q}(E) \longrightarrow X$.  
\begin{theorem}  [Darondeau--Pragacz formula  of type $A$ in complex cobordism]  
\label{thm:TypeAD-PFormula(ComplexCobordism)}    
For a polynomial
 $f(t_{1}, \ldots, t_{q})  \in MU^{*}(X)[t_{1}, \ldots, t_{q}]^{S_{q_{1}} \times S_{q_{2} - q_{1}} \times \cdots \times S_{q - q_{m-1}}}$,  
one has 
\begin{equation}   \label{eqn:TypeAD-PFormula(ComplexCobordism)} 
\begin{array}{llll}   
   &  \varpi_{q_{1}, \ldots, q_{m-1}, q *}  (f(y_{1}, \ldots, y_{q}))   \medskip \\
    &   =   \left [ \displaystyle{\prod_{k=1}^{m}}  
                   \prod_{q_{k-1} < i \leq q_{k}}  t_{i}^{(n-1) - (q_{k}-i)}   \right ]  
   \left (  
               f(t_{1}, \ldots, t_{q})  \times \displaystyle{\prod_{k=1}^{m}}  
               s^{\L}_{\emptyset} (t_{q_{k-1} + 1},  \ldots, t_{q_{k}})^{-1} 
             \prod_{1 \leq i < j \leq q}  (t_{j} +_{\L}  \overline{t}_{i})  
   \right.   \medskip \\
    &     \hspace{11.2cm}   \left.   \times    \;         \displaystyle{\prod_{i=1}^{q}}  \mathscr{S}^{\L} (E^{\vee};  1/t_{i}) 
  \right ).    \medskip 
\end{array}  
\end{equation} 
\end{theorem}  
Before starting the proof, we recall the following fact concerning the 
{\it universal Schur class} of a vector bundle: 
For the Gysin map:  $\varpi_{*}:  MU^{*}(\F \ell (E)) \longrightarrow   MU^{*}(X)$, 
the following formula holds (see Nakagawa--Naruse \cite[Corollary 4.8]{Nakagawa-Naruse2018}): 
\begin{equation*} 
   \varpi_{*}  (\bm{y}^{\lambda + \rho_{n-1}})   = s^{\L}_{\lambda}(E^{\vee}),  
\end{equation*} 
where $s^{\L}_{\lambda}(E^{\vee})  \in MU^{2|\lambda|} (X)$ is a cohomology
class defined by $\varpi^{*} (s^{\L}_{\lambda}(E^{\vee})) =  s^{\L}_{\lambda}(\bm{y}_{n})$.
Hereafter,  we  often  identify  $s^{\L}_{\lambda}(E^{\vee})$ with $s^{\L}_{\lambda}(\bm{y}_{n})$  via 
 monomorphism $\varpi^{*}:  MU^{*}(X) \longrightarrow MU^{*}(\F \ell (E))$, and write 
$s^{\L}_{\lambda}(E^{\vee}) = s^{\L}_{\lambda}(\bm{y}_{n})$.   
As a particular case, for $\lambda = \emptyset$, 
we have 
\begin{equation}  \label{eqn:varpi_*(y^rho_n-1)}  
    \varpi_{*}  (\bm{y}^{\rho_{n-1}})  = s^{\L}_{\emptyset}(\bm{y}_{n}).   
\end{equation} 
Since we know that 
$s^{\L}_{\emptyset} (\bm{y}_{n})  =  (1 +  \text{higher terms in}  \; y_{1}, \ldots, y_{n})   \in MU^{0}(X)$ 
is an invertible element,  we deduce that 
\begin{equation*}   \label{eqn:varpi_*(-)=1}  
   \varpi_{*}   (s^{\L}_{\emptyset}(\bm{y}_{n})^{-1}  \bm{y}^{\rho_{n-1}}) = 1. 
\end{equation*}

\begin{proof}[Proof of Theorem $\ref{thm:TypeAD-PFormula(ComplexCobordism)}$]
One can prove the theorem along the same lines as in 
Darondeau--Pragacz \cite[Theorem 1.1]{Darondeau-Pragacz2017}. 
First,  we prove the case of full flag bundles
 $\varpi_{1, 2, \ldots, q}:  \F \ell_{1, 2, \ldots, q}(E)  \longrightarrow X$ 
by  induction on $q \geq 1$.   For the case $q  = 1$, 
the result is simply the formula (\ref{eqn:FundamentalFormula(TypeA)}).  
Hence,   we may assume the result  for the case  of  $q - 1$ with  $q \geq 2$.  
Thus,  for any polynomial $g(t_{1}, \ldots, t_{q - 1}) \in MU^{*}(X)[t_{1}, \ldots, t_{q-1}]$, 
one has 
\begin{equation}      \label{eqn:TypeAD-PFormula(InductionAssumption)}  
\begin{array}{lll} 
  &   \varpi_{1, \ldots, q - 1 *} (g(y_{1},  \ldots, y_{q - 1}))   \medskip \\
    &  =  [t_{1}^{n-1}  \cdots t_{q-1}^{n-1}] 
        \left (   
                   g(t_{1}, \ldots, t_{q-1})  \displaystyle{\prod_{1 \leq i < j \leq q-1}}  (t_{j} +_{\L} \overline{t}_{i}) 
                     \displaystyle{\prod_{i=1}^{q-1}}   \mathscr{S}^{\L} (E^{\vee}; 1/t_{i})  
         \right ).   \medskip 
\end{array}    
\end{equation} 

Now,  we consider the image of the  Gysin map $\varpi_{1, \ldots, q *} (f(y_{1}, \ldots, y_{q}))$. 
Since $\varpi_{1, \ldots, q}:  \F \ell_{1, \ldots, q}(E) \longrightarrow X$ 
is the composite of $\varpi_{q}: \F \ell_{1, \ldots, q}(E) \longrightarrow \F \ell_{1, \ldots, q-1}(E)$ 
and $\varpi_{1, \ldots, q-1}: \F \ell_{1, \ldots, q-1}(E) \longrightarrow X$, namely, 
$\varpi_{1, \ldots, q}   =  \varpi_{1, \ldots, q-1} \circ \varpi_{q}$,  we have 
\begin{equation*} 
    \varpi_{1, \ldots, q *} (f(y_{1}, \ldots, y_{q}))  = 
    \varpi_{1, \ldots, q-1 *} \circ  \varpi_{q *} (f(y_{1}, \ldots, y_{q})). 
\end{equation*} 
By the construction recalled in \S \ref{subsec:FlagBundles},  
$\varpi_{q}:  \F \ell_{1, \ldots, q}(E) \longrightarrow \F \ell_{1, \ldots, q-1}(E)$ is 
the same as the projective bundle of lines $\varpi_{q}:  P(E/U_{q-1})  \longrightarrow \F \ell_{1, \ldots, q-1}(E)$.  
The rank of the quotient bundle $E/U_{q-1}$ 
is $n - q + 1$, and therefore,  by the fundamental formula (\ref{eqn:FundamentalFormula(TypeA)}), we have   
\begin{equation*} 
     \varpi_{q *} (f(y_{1}, \ldots, y_{q-1}, y_{q})) 
   =  [t_{q}^{n-q}]    (f(y_{1}, \ldots, y_{q-1}, t_{q}) \mathscr{S}^{\L} ((E/U_{q-1})^{\vee}; 1/t_{q})). 
\end{equation*} 
Here,   the vector bundle $(E/U_{q-1})^{\vee}$ has the Chern roots 
$y_{q}, \ldots, y_{n}$ as is easily seen by the definition of the Chern roots of $E^{\vee}$,  
  and hence we deduce from (\ref{eqn:SegreSeries(ComplexCobordism)}),  
\begin{equation*} 
     \mathscr{S}^{\L}  ((E/U_{q-1})^{\vee}; 1/t_{q})   
    =  \dfrac{t_{q}^{n - q + 1}}  
                {\mathscr{P}^{\L}(t_{q})   
                  \prod_{i= q}^{n}  (t_{q} +_{\L} \overline{y}_{i})} 
    =  t_{q}^{-(q-1)}  \prod_{i=1}^{q-1} (t_{q} +_{\L} \overline{y}_{i}) 
      \mathscr{S}^{\L} (E^{\vee};  1/t_{q}). 
\end{equation*} 
Thus,  we have 
\begin{equation*}
\begin{array}{llll}   
   \varpi_{q *}  (f(y_{1}, \ldots, y_{q-1}, y_{q})) 
 &=   [t_{q}^{n-q}]   
      \left ( f(y_{1}, \ldots, y_{q-1}, t_{q})  
          t_{q}^{-(q-1)}   \displaystyle{\prod_{i=1}^{q-1}}  (t_{q} +_{\L} \overline{y}_{i} )  
         \mathscr{S}^{\L} (E^{\vee};  1/t_{q})  
       \right )        \medskip \\
 &  = [t_{q}^{n-1}]  \left (  
                            f(y_{1}, \ldots, y_{q-1},  t_{q})   \,  
                         \displaystyle{\prod_{i=1}^{q-1}}  (t_{q} +_{\L} \overline{y}_{i})   \, 
                   \mathscr{S}^{\L}   (E^{\vee};  1/t_{q})   
                      \right ),    \medskip 
\end{array}  
\end{equation*} 
and hence, 
\begin{equation*} 
\begin{array}{llll} 
    \varpi_{1, \ldots, q *}(f(y_{1}, \ldots, y_{q})) 
             & =  [t_{q}^{n-1}] 
                  \left [   
                            \varpi_{1, \ldots, q-1 *} 
                   \left (  
                            f(y_{1}, \ldots, y_{q-1},  t_{q})   \,  
                         \displaystyle{\prod_{i=1}^{q-1}}  (t_{q} +_{\L} \overline{y}_{i})   \, 
                   \mathscr{S}^{\L}   (E^{\vee}; 1/t_{q})   
                      \right )
                   \right ]    \medskip \\
          & =  [t_{q}^{n-1}] 
                  \left [   
                            \varpi_{1, \ldots, q-1 *} 
                   \left (  
                            f(y_{1}, \ldots, y_{q-1},  t_{q})   \,  
                         \displaystyle{\prod_{i=1}^{q-1}}  (t_{q} +_{\L} \overline{y}_{i})   
                   \right )  
                   \mathscr{S}^{\L}   (E^{\vee}; 1/t_{q})  
                   \right ].     \medskip  
\end{array}  
\end{equation*} 
Then,  by the induction assumption (\ref{eqn:TypeAD-PFormula(InductionAssumption)}), 
we have 
\begin{equation*} 
\begin{array}{llll}  
 &  \varpi_{1, \ldots, q-1 *} 
                   \left (  
                            f(y_{1}, \ldots, y_{q-1},  t_{q})   \,  
                         \displaystyle{\prod_{i=1}^{q-1}}  (t_{q} +_{\L} \overline{y}_{i})   
                   \right )    \medskip \\
& = [t_{1}^{n-1} \cdots t_{q-1}^{n-1}]  
      \left (   
          f(t_{1}, \ldots, t_{q-1}, t_{q})   \displaystyle{\prod_{i=1}^{q-1}} (t_{q}  +_{\L}  \overline{t}_{i}) 
                   \times \prod_{1 \leq i <  j \leq q-1} (t_{j} +_{\L} \overline{t}_{i}) 
                   \prod_{i=1}^{q-1} \mathscr{S}^{\L}  (E^{\vee}; 1/t_{i})  
       \right ),      \medskip 
\end{array}  
\end{equation*} 
and therefore,    we obtain the desired formula.

From the result of full flag bundles, we can prove the 
case of general partial flag bundles $\varpi_{q_{1}, \ldots, q_{m}}:  \F \ell_{q_{1}, \ldots, q_{m}}(E) 
\longrightarrow X$.   For simplicity,  we set $\F := \F \ell_{q_{1}, \ldots, q_{m}}(E)$, and 
$q := q_{m}$.   On $\F$, we have the universal flag of subbundles of $E$: 
\begin{equation*} 
   U_{q_{0}} = 0  \subsetneq  U_{q_{1}}  \subsetneq \cdots \subsetneq  U_{q_{m-1}}  \subsetneq 
   U_{q} \subsetneq  U_{q_{m + 1}} =  E.  
\end{equation*}
Let us consider the  fiber product  
\begin{equation*} 
  \mathcal{Y}:=  \F \ell (U_{q_{1}}) \times_{\F}  \F \ell (U_{q_{2}}/U_{q_{1}})  \times_{\F}  
         \cdots \times_{\F}  \F \ell (U_{q}/U_{q_{m-1}})  
\end{equation*} 
with the natural projection map  $\varpi_{\F}:  \mathcal{Y}  \longrightarrow \F$. 
Then, by the definition of the full flag bundles 
$\varpi_{k}':  \F \ell (U_{q_{k}}/U_{q_{k-1}})   \longrightarrow \F$ ($k = 1, 2, \ldots, m$), 
the variety $\mathcal{Y}$  is  naturally isomorphic to $\F \ell_{1, 2, \ldots, q}(E)$.  
We denote  this isomorphism by $\theta:  \mathcal{Y}  \overset{\sim}{\longrightarrow} \F \ell_{1, 2, \ldots, q}(E)$.  
Identifying $\mathcal{Y}$ with $\F \ell_{1, 2, \ldots, q}(E)$ through $\theta$, 
we have $\varpi_{1, 2, \ldots, q}  =  \varpi_{q_{1}, \ldots, q_{m-1}, q} \circ \varpi_{\F}$. 
Therefore, by the naturality of the Gysin map, we have 
   $\varpi_{1, 2, \ldots, q *}  =  \varpi_{q_{1}, \ldots, q_{m-1}, q *} \circ \varpi_{\F *}$.  
By applying the formula (\ref{eqn:varpi_*(y^rho_n-1)}) to the full flag 
bundles $\varpi_{k}': \F \ell (U_{q_{k}}/U_{q_{k-1}}) \longrightarrow \F$, 
we  have 
\begin{equation*} 
     \varpi_{k *}' \left ( \prod_{q_{k-1} < i \leq q_{k}}  y_{i}^{q_{k} - i} \right ) 
    = s^{\L}_{\emptyset}  ((U_{q_{k}}/U_{q_{k-1}})^{\vee}) 
    =  s^{\L}_{\emptyset} (y_{q_{k-1} + 1},  \ldots, y_{q_{k}}),  
\end{equation*} 
and hence we obtain 
\begin{equation*} 
   \varpi_{\F *}  \left ( \prod_{k=1}^{m}  \prod_{q_{k-1} < i \leq q_{k}} y_{i}^{q_{k} - i}  \right ) 
        = \prod_{k=1}^{m}  s^{\L}_{\emptyset}  (y_{q_{k-1} + 1},  \ldots, y_{q_{k}}).  
\end{equation*} 
Since  each $s^{\L}_{\emptyset} (y_{q_{k-1} + 1},  \ldots, y_{q_{k}})$ is an invertible 
element, we have 
\begin{equation}   \label{eqn:varpi_F(-)=1}  
    \varpi_{\F *}  \left (
                  \prod_{k=1}^{m}  s^{\L}_{\emptyset}  (y_{q_{k-1} + 1}, \ldots, y_{q_{k}})^{-1} 
                          \prod_{k=1}^{m}  \prod_{q_{k-1} < i \leq q_{k}} y_{i}^{q_{k} - i}  \right ) 
       =  1. 
\end{equation} 
With these preliminaries, we proceed with the proof as follows: 
For a polynomial $f(t_{1},  t_{2}, \ldots, t_{q})  \in 
MU^{*}(X)[t_{1}, t_{2},  \ldots, t_{q}]^{S_{q_{1}} \times S_{q_{2} - q_{1}} \times \cdots \times 
S_{q - q_{m-1}}}$,  we compute 
\begin{equation*} 
\begin{array}{llll} 
  &  \varpi_{q_{1}, \ldots, q_{m-1}, q  *} (f (y_{1}, y_{2}, \ldots, y_{q}))   \medskip \\
 & =      \varpi_{q_{1}, \ldots, q_{m-1}, q *} (f(y_{1}, \ldots, y_{q}) \times 1)  \medskip \\
 & =     \varpi_{q_{1}, \ldots, q_{m-1}, q *} 
               \left (f (y_{1}, \ldots, y_{q})  \times  
                 \varpi_{\F *}  
                         \left ( 
                               \displaystyle{
                                           \prod_{k=1}^{m} 
                             s^{\L}_{\emptyset}  (y_{q_{k-1} + 1}, \ldots, y_{q_{k}})^{-1} \prod_{k=1}^{m}} 
                                                         \prod_{q_{k-1} < i \leq q_{k}}   y_{i}^{q_{k} - i}
                            \right )  
               \right )   
\medskip \\  
   & =  \varpi_{q_{1}, \ldots, q_{m-1}, q *} \circ \varpi_{\F *} 
              \left (  f (y_{1}, \ldots, y_{q})  \times    
                          \displaystyle{\prod_{k=1}^{m}}
                     s^{\L}_{\emptyset}  (y_{q_{k-1} + 1}, \ldots, y_{q_{k}})^{-1}      
        \displaystyle{\prod_{k=1}^{m}}  \prod_{q_{k-1} < i \leq q_{k}}   y_{i}^{q_{k} - i} 
               \right )  
 \medskip \\
& =  \varpi_{1, 2, \ldots, q *} 
         \left (  f (y_{1}, \ldots, y_{q})  \times    
                             \displaystyle{\prod_{k=1}^{m}} 
               s^{\L}_{\emptyset}  (y_{q_{k-1} + 1},  \ldots, y_{q_{k}})^{-1}   
       \displaystyle{\prod_{k=1}^{m}}  \prod_{q_{k-1} < i \leq q_{k}}   y_{i}^{q_{k} - i}
               \right )  
 \medskip \\
 & = \left [ \displaystyle{\prod_{i=1}^{q}} t_{i}^{n-1} \right ] 
     \left (   f(t_{1}, \ldots, t_{q})    \times 
                  \displaystyle{\prod_{k=1}^{m}}  s^{\L}_{\emptyset}  (t_{q_{k-1} + 1},  \ldots, t_{q_{k}})^{-1}    
       \displaystyle{\prod_{k=1}^{m}}  \prod_{q_{k-1} < i \leq q_{k}}   t_{i}^{q_{k} - i}   \right. \medskip \\  
      &  \hspace{8cm}   \left.  
                       \times  \displaystyle{\prod_{1 \leq i < j \leq q}}  (t_{j} +_{\L}  \overline{t}_{i})    
        \times     \displaystyle{\prod_{i=1}^{q}}   \mathscr{S}^{\L} (E^{\vee}; 1/t_{i})   
        \right ) \medskip \\ 
& =  \left [  \displaystyle{
                                      \prod_{k=1}^{m}  \prod_{q_{k-1} < i \leq q_{k}} 
                                 }  
                                                             t_{i}^{(n-1) - (q_{k} - i)}   
          \right ] 
        \left (  f(t_{1}, \ldots, t_{q})  \times 
                    \displaystyle{\prod_{k=1}^{m}}  s^{\L}_{\emptyset}  (t_{q_{k-1} + 1}, \ldots, t_{q_{k}})^{-1}   
     \displaystyle{\prod_{1 \leq i < j \leq q}}  (t_{j} +_{\L} \overline{t}_{i})    \right. \medskip \\ 
   &  \hspace{11cm}   \left.                          
                          \times  \displaystyle{\prod_{i=1}^{q}}  \mathscr{S}^{\L}  (E^{\vee}; 1/t_{i})   
         \right ).  \medskip 
\end{array} 
\end{equation*} 
\end{proof}

In $K$-theory, since $G_{\emptyset} (-) = 1$, the above D--P formula takes 
the following form$:$  
\begin{cor}  [Darondeau--Pragacz formula of type $A$ in $K$-theory]
      \begin{equation}   \label{eqn:TypeAD-PFormula(K-theory)} 
\begin{array}{llll}
   &  \varpi_{q_{1}, \ldots ,q_{m-1}, q *}  (f(y_{1}, \ldots, y_{q}))   \medskip \\
    &   =   \left [ \displaystyle{\prod_{k=1}^{m}}  
                 \prod_{q_{k-1} < i \leq q_{k}}  t_{i}^{(n-1) - (q_{k}-i)}   \right ]  
   \left (  
               f(t_{1}, \ldots, t_{q})  \times   \displaystyle{\prod_{1 \leq i < j \leq q}} 
                 (t_{j} \ominus t_{i})  
         \;         \displaystyle{\prod_{i=1}^{q}}  G(E^{\vee};  1/t_{i}) 
  \right ).    \medskip 
\end{array}  
\end{equation} 
\end{cor}

\begin{rem}  \label{rem:SymmetrizingOperatorDescription(TypeA)}  
It is known that Gysin maps for various flag bundles 
 can be described as  certain  symmetrizing operators
$($see Bressler--Evens \cite[Theorem 1.8]{Bressler-Evens1990},   
Nakagawa--Naruse   \cite[Theorem 2.5, Corollary 2.6]{Nakagawa-Naruse2018} for complex cobordism$)$.  
For the partial flag bundle  $\varpi_{q_{1}, \ldots, q_{m} *}: MU^{*}(\F \ell_{q_{1}, \ldots, q_{m}}(E))  \longrightarrow MU^{*}(X)$, 
the formula is given as follows$:$ 
For a symmetric polynomial 
$f(t_{1}, \ldots, t_{n})  \in MU^{*}(X)[t_{1}, \ldots, t_{n}]^{S_{q_{1}} \times S_{q_{2}-q_{1}} \times \cdots \times S_{n - q_{m}}}$,  
one has 
\begin{equation}
\begin{array}{lll} 
 &   \varpi_{q_{1}, \ldots, q_{m} *} (f(y_{1}, \ldots, y_{n})) \medskip \\
  &  =   \displaystyle{
                               \sum_{\overline{w}  \in S_{n}/S_{q_{1}} \times S_{q_{2} - q_{1}} 
                                                                \times \cdots \times S_{n - q_{m}}
                                       }
                           }   w \cdot \left [ 
                      \dfrac{ f(y_{1}, \ldots, y_{n}) } 
                              {  \prod_{k=1}^{m}
                                                           \prod_{ q_{k-1} < i \leq q_{k} 
                                                                    } 
                                                             \prod_{
                                                                                      q_{k} < j \leq n  
                                                                                    }  
                                          (y_{i}  +_{\L}  \overline{y}_{j})  
                              } 
                       \right ].     \medskip 
\end{array} 
\end{equation} 
\end{rem} 

\subsection{Darondeau--Pragacz formula of type $C$ in complex cobordism}  
The same technique can be applied to the type $C$ case because 
type $C$ flag bundles are also constructed as a chain of projective bundles of lines. 
Indeed,  we can extend the Darondeau--Pragacz formula 
\cite[Theorem 2.1]{Darondeau-Pragacz2017} in the following manner: 
Let $E  \longrightarrow X$ be a symplectic vector bundle of rank $2n$. 
Given a sequence of integers $q_{0} = 0< q_{1} < \cdots<  q_{m} \leq n$,  
we set $q := q_{m}$.  
Then, the following Gysin formula holds  for the isotropic partial flag bundle
$\varpi_{q_{1}, \ldots, q_{m-1}, q} =\varpi^{C}_{q_{1}, \ldots, q_{m-1}, q}:  \F \ell^{C}_{q_{1}, \ldots, q_{m-1}, q}(E) 
\longrightarrow X$.    
\begin{theorem}   [Darondeau--Pragacz formula of type $C$ in complex cobordism] \label{thm:TypeCD-PFormula(ComplexCobordism)} 
For a polynomial 
$f (t_{1}, \ldots, t_{q})   \in MU^{*}(X)[t_{1}, \ldots, t_{q}]^{S_{q_{1}} \times S_{q_{2} - q_{1}} \times \cdots  \times S_{q - q_{m-1}}}$, 
one has 
\begin{equation*} 
\begin{array}{lll} 
  & \varpi_{q_{1}, \ldots, q_{m-1}, q *}(f(y_{1}, \ldots, y_{q}))    \medskip \\
 &    =  \left [  \displaystyle{\prod_{k=1}^{m}}   \prod_{q_{k-1} < i \leq q_{k}}   t_{i}^{(2n - 1) - (q_{k} - i)} 
           \right ]  
    \left (  
            f(t_{1}, \ldots, t_{q})  
            \times \displaystyle{\prod_{k=1}^{m}}  
                s^{\L}_{\emptyset} (t_{q_{k-1} + 1}, \ldots, t_{q_{k}})^{-1}   \right.  \medskip \\
  &  \hspace{6cm}  
      \left.  
       \times   \displaystyle{\prod_{1 \leq i < j \leq q}} 
        (t_{j} +_{\L}  \overline{t}_{i})  (t_{j} +_{\L} t_{i} +_{\L}   z)  
           \prod_{i=1}^{q}   {\mathscr{S}}^{\L}  (E^{\vee}; 1/t_{i}) 
    \right ).    \medskip 
\end{array}  
\end{equation*} 
\end{theorem}  

\begin{proof} 
As in the proof due to  Darondeau--Pragacz \cite[Theorem 2.1]{Darondeau-Pragacz2017}, 
  we first consider the case of isotropic full flag bundles.  
We shall  prove the formula for $\varpi_{1, 2, \ldots, q}:  \F \ell^{C}_{1, 2, \ldots, q} (E) 
\longrightarrow X$   by induction on $q \geq 1$.   
For the case $q = 1$, the result follows from the formula 
(\ref{eqn:FundamentalFormula(TypeA)}) since $\F \ell^{C}_{1}(E)  = P(E)$. 
Hence,   we may assume the result for the case of $q-1$ with $q \geq 2$.  
Let us consider the projection $\varpi_{q}:  \F \ell^{C}_{1, \ldots, q}(E)  
\longrightarrow \F \ell^{C}_{1, \ldots, q-1}(E)$, which is the same as 
the projective bundle of lines $\varpi_{q}:  P(U_{q-1}^{\perp}/U_{q-1})  \longrightarrow 
\F \ell^{C}_{1, \ldots, q-1}(E)$.   Here,  the rank of the quotient bundle $U_{q-1}^{\perp}/U_{q-1}$ 
is  $2(n-q) + 2$.  Therefore, for a polynomial $f(t_{1}, \ldots, t_{q})  \in MU^{*}(X)[t_{1}, \ldots, t_{q}]$,  
the fundamental formula (\ref{eqn:FundamentalFormula(TypeA)}) gives 
\begin{equation*} 
    \varpi_{q *}(f(y_{1}, \ldots, y_{q-1}, y_{q})) 
     = [t_{q}^{2(n-q) + 1}] 
       \left (  f(y_{1}, \ldots, y_{q-1},  t_{q})  
      \mathscr{S}^{\L} ( (U_{q-1}^{\perp}/U_{q-1})^{\vee}; 1/t_q) \right ).  
\end{equation*} 
Since the vector bundle $(U_{q-1}^{\perp}/U_{q-1})^{\vee}$ has the 
Chern roots $y_{q}, \ldots, y_{n},   \overline{y}_{q} +_{\L}  \overline{z}, 
\ldots, \overline{y}_{n} +_{\L}  \overline{z}$ as seen in  \S  \ref{subsec:FlagBundles}, 
we deduce from (\ref{eqn:SegreSeries(ComplexCobordism)}),  
\begin{equation*} 
\begin{array}{llll}  
  \mathscr{S}^{\L} ((U_{q-1}^{\perp}/U_{q-1})^{\vee}; 1/t_{q}) 
   &   =   \dfrac{t_{q}^{2(n-q) + 2}}  
{ \mathscr{P}^{\L} (t_{q})    
           \prod_{i=q}^{n}  (t_{q} +_{\L}  \overline{y}_{i}) 
                              (t_{q} +_{\L}  y_{i} +_{\L} z) }  \medskip \\
    & =    t_{q}^{-2(q-1)}  \displaystyle{\prod_{i=1}^{q-1}}  
                                                     (t_{q} +_{\L} \overline{y}_{i})  
                                                     (t_{q} +_{\L} y_{i} +_{\L}  z) 
                 \mathscr{S}^{\L}  (E^{\vee}; 1/t_{q}).    \medskip  
\end{array}  
\end{equation*} 
Thus, we have 
\begin{equation*} 
\begin{array}{llll}  
&    \varpi_{q *}  (f(y_{1}, \ldots, y_{q-1}, y_{q}))   \medskip \\
 & =   [t_{q}^{2(n-q) + 1}] 
        \left  ( 
          f(y_{1}, \ldots, y_{q-1},  t_{q})  
         t_{q}^{-2(q-1)}  \displaystyle{\prod_{i=1}^{q-1}}  
                                                     (t_{q} +_{\L} \overline{y}_{i})  
                                                     (t_{q} +_{\L} y_{i} +_{\L}  z) 
                 \mathscr{S}^{\L} (E^{\vee}; 1/t_{q})  
         \right )   \medskip \\
 & =    [t_{q}^{2n-1}] 
        \left ( 
            f(y_{1}, \ldots, y_{q-1},  t_{q})  
           \displaystyle{\prod_{i=1}^{q-1}}  
                                                     (t_{q} +_{\L} \overline{y}_{i})  
                                                     (t_{q} +_{\L} y_{i} +_{\L}  z) 
                 \mathscr{S}^{\L} (E^{\vee}; 1/t_{q}) 
         \right ).        \medskip 
\end{array} 
\end{equation*} 
Therefore, we have 
\begin{equation*} 
\begin{array}{llll}  
 &  \varpi_{1, \ldots, q *} (f(y_{1}, \ldots, y_{q})) 
   =  \varpi_{1, \ldots, q-1 *} \circ \varpi_{q *} (f(y_{1}, \ldots, y_{q}))      \medskip \\
  & =  [t_{q}^{2n-1}]  
           \left [     \varpi_{1, \ldots, q-1 *}  
                       \left (  
                            f(y_{1}, \ldots, y_{q-1},  t_{q})  
           \displaystyle{\prod_{i=1}^{q-1}}  
                                                     (t_{q} +_{\L} \overline{y}_{i})  
                                                     (t_{q} +_{\L} y_{i} +_{\L}  z) 
                 \mathscr{S}^{\L} (E^{\vee}; 1/t_{q})
                        \right ) 
            \right ]    \medskip \\
    & =  [t_{q}^{2n-1}]  
           \left [     \varpi_{1, \ldots, q-1 *}  
                       \left (  
                            f(y_{1}, \ldots, y_{q-1},  t_{q})  
           \displaystyle{\prod_{i=1}^{q-1}}  
                                                     (t_{q} +_{\L} \overline{y}_{i})  
                                                     (t_{q} +_{\L} y_{i} +_{\L}  z) 
                      \right )
                 \mathscr{S}^{\L} (E^{\vee}; 1/t_{q})                       
            \right ].      \medskip
\end{array} 
\end{equation*} 
Then, by the induction assumption, we obtain the required formula.

From the result of isotropic full flag bundles, we can prove the case of 
isotropic partial flag bundles
 $\varpi_{q_{1}, \ldots, q_{m}}:  \F \ell^{C}_{q_{1}, \ldots, q_{m}}(E) \longrightarrow X$
by the same manner as the type $A$ case.  The space $\mathcal{Y}$, that is the 
type $C$ analogue  of the space used in the proof of  Theorem \ref{thm:TypeAD-PFormula(ComplexCobordism)}, 
is  naturally isomorphic to $\F \ell^{C}_{1, 2, \ldots, q}(E)$  because 
any flag inside an isotropic subbundle is also an isotropic flag.  
\end{proof} 

In $K$-theory, the above D--P formula takes 
the following form$:$  
\begin{cor}  [Darondeau--Pragacz formula of type $C$ in $K$-theory]
      \begin{equation}   \label{eqn:TypeCD-PFormula(K-theory)} 
\begin{array}{llll}
   &  \varpi_{q_{1}, \ldots ,q_{m-1}, q *}  (f(y_{1}, \ldots, y_{q}))   \medskip \\
    &   =   \left [ \displaystyle{\prod_{k=1}^{m}}  
                 \prod_{q_{k-1} < i \leq q_{k}}  t_{i}^{(2n-1) - (q_{k}-i)}   \right ]  
   \left (  
               f(t_{1}, \ldots, t_{q})  \times   \displaystyle{\prod_{1 \leq i < j \leq q}} 
                 (t_{j} \ominus t_{i})(t_{j} \oplus t_{i} \oplus z)    \right.  \medskip \\
  &  \hspace{10cm}  \left.     \times   \;        \displaystyle{\prod_{i=1}^{q}}  G(E^{\vee};  1/t_{i}) 
  \right ).    \medskip 
\end{array}  
\end{equation} 
\end{cor} 

\subsection{Darondeau--Pragacz formulas of types $B$ and $D$ in complex cobordism}  

\subsubsection{Fundamental Gysin formula for quadric bundle in complex cobordism}  
Let $E  \longrightarrow    X$ be an orthogonal vector bundle of  rank $N$ ($2n + 1$ for $Y = B$,  or $2n$ for $Y = D$), 
and let $\rho_{1}:  Q(E)  \longrightarrow X$  be the associated quadric bundle of 
isotropic lines in $E$.   Denote the tautological line bundle on 
$Q(E)$ by $U_{1}$.      Put  $y_{1} = c^{MU}_{1}(U_{1}^{\vee})  \in MU^{2}(Q(E))$ 
(We also  denote the tautological line bundle on $P(E)$ by $U_{1}$, and we use 
the same symbol $y_{1}$ for $c^{MU}_{1}(U_{1}^{\vee})  \in MU^{2}(P(E))$. Thus,  
$\iota^{*}(y_{1}) = y_{1}$).  
Following Darondeau--Pragacz  \cite[\S 3.3]{Darondeau-Pragacz2017},  
we shall describe the Gysin homomorphism  $\rho_{1 *}:  MU^{*}(Q(E))  \longrightarrow 
MU^{*}(X)$.     First,  we shall describe the class $[Q(E)]  \in MU^{2}(P(E))$. 
By the definition of $Q(E)$,  it is given by the zero set of a section 
of the line bundle
    $\Hom \,  (U_{1}  \otimes U_{1}, L) \cong U_{1}^{\vee} \otimes U_{1}^{\vee} \otimes L$. 
Therefore  we have 
\begin{equation*} 
       [Q(E)]  =  c^{MU}_{1}(U_{1}^{\vee}  \otimes U_{1}^{\vee}  \otimes L) 
                 =    y_{1} +_{\L}  y_{1} +_{\L} z  =  [2]_{\L}(y_{1})  +_{\L} z. 
\end{equation*} 
Let $\varpi_{1}: P(E)  \longrightarrow X$ be the projective bundle of lines,  
and $\iota: Q(E)  \hooklongrightarrow P(E)$  be the natural inclusion. 
Then,  we have $\rho_{1} = \varpi_{1} \circ \iota$, and therefore  
$\rho_{1 *}  =  \varpi_{1 *}  \circ \iota_{*}:  MU^{*}(Q(E))  \longrightarrow MU^{*}(X)$. 
Then, using the  fundamental formula (\ref{eqn:FundamentalFormula(TypeA)}), 
which  still  holds for a formal  power series $f(t)   \in MU^{*}(X)[[t]]$,    
one can compute for $k \geq 0$, 
\begin{equation*} 
\begin{array}{llll}  
    \rho_{1 *} (y_{1}^{k})  
 & =  \varpi_{1 *}  \circ  \iota_{*} (y_{1}^{k})  
  =  \varpi_{1 *} \circ \iota_{*}(\iota^{*}(y_{1}^{k}))  
  =  \varpi_{1 *}(y_{1}^{k}  \, \iota_{*} (1) )
 =   \varpi_{1 *} (y_{1}^{k}  [Q(E)])    \medskip \\
 & = \varpi_{1 *} (y_{1}^{k}  ([2]_{\L}(y_{1})  +_{\L} z)) 
   =  [t^{N-1}]  (t^{k}  ([2]_{\L}(t)  +_{\L} z)  \mathscr{S}^{\L}(E^{\vee}; 1/t)).   \medskip  
\end{array}  
\end{equation*}  
Therefore, for a polynomial $f(t)  \in MU^{*}(X)[t]$,  the Gysin homomorphism 
$\rho_{1 *}:  MU^{*}(Q(E))   \\ \longrightarrow  MU^{*}(X)$ is given by 
\begin{equation}    \label{eqn:FundamentalFormula(TypeBD)}  
    \rho_{1 *}  (f(y_{1}))   =  [t^{N-1}]  
                                      (f(t) ([2]_{\L}(t)  +_{\L} z)  \mathscr{S}^{\L} (E^{\vee}; 1/t) 
                                      ).  
\end{equation} 
This is the fundamental formula for establishing more general Gysin 
formulas for  flag bundles of types $B$ and $D$,  
which will be given in the next subsection.

\subsubsection{Darondeau--Pragacz formula of types $B$ and  $D$  in complex cobordism} 
With the above preliminaries, we can extend  the Darondeau--Pragacz formula \cite[Theorem 3.1]{Darondeau-Pragacz2017})
  in the following manner: Let $E \longrightarrow X$ be an orthogonal 
vector bundle of rank $N$.   Given a sequence of 
integers $q_{0} = 0 < q_{1} < \cdots < q_{m} \leq n$, 
we set $q := q_{m}$.  
Then, the following Gysin formula holds for the isotropic partial flag bundle 
$\rho_{q_{1}, \ldots, q_{m-1}, q}  = \varpi^{Y}_{q_{1}, \ldots, q_{m-1}, q}  : \F \ell^{Y}_{q_{1}, \ldots, q_{m-1}, q}(E) \longrightarrow X$. 
\begin{theorem} [Darondeau--Pragacz formula of types $B$ and $D$ in complex cobordism] 
\label{thm:TypeBDD-PFormula(ComplexCobordism)}    
For a polynomial 
$f(t_{1}, \ldots, t_{q})  \in MU^{*}(X)[t_{1}, \ldots, t_{q}]^{S_{q_{1}} \times S_{q_{2} - q_{1}} \times \cdots \times S_{q - q_{m-1}}}$, 
one has  
\begin{equation*} 
\begin{array}{lll} 
   &  \rho_{q_{1}, \ldots, q_{m-1}, q *} (f(y_{1}, \ldots, y_{q}))    \medskip \\
   & =  \left [ \displaystyle{\prod_{k=1}^{m}} \prod_{q_{k-1} < i \leq q_{k}}  
                 t_{i}^{(N-1) - (q_{k} - i)}  \right ]  
      \left ( 
              f(t_{1}, \ldots, t_{q})   \times \displaystyle{\prod_{i=1}^{q}}   ([2]_{\L}(t_{i}) +_{\L} z)   \times \displaystyle{\prod_{k=1}^{m}}   s^{\L}_{\emptyset} (t_{q_{k-1} + 1},  \ldots, t_{q_{k}})^{-1}  \right. \medskip \\ 
      &  \hspace{6.5cm}  \left.        \times   \displaystyle{\prod_{1 \leq i < j \leq q}}  (t_{j} +_{\L}  \overline{t}_{i})    (t_{j} +_{\L} t_{i} +_{\L} z)  
                          \displaystyle{\prod_{i=1}^{q}}   \mathscr{S}^{\L} (E^{\vee}; 1/t_{i})  
                                         \right ).    \medskip 
\end{array} 
\end{equation*}
\end{theorem} 

\begin{proof} 
One can prove the theorem in the same manner as the type $C$ case, 
replacing the fundamental formula (\ref{eqn:FundamentalFormula(TypeA)}) for a 
projective bundle  with 
the fundamental formula (\ref{eqn:FundamentalFormula(TypeBD)})
for a quadric bundle.   For example, the first step of the induction 
proceeds as follows:  The projection $\rho_{q}:  \F \ell^{Y}_{1, \ldots, q} (E) \longrightarrow 
\F \ell^{Y}_{1, \ldots, q-1}(E)$  is the same as the quadric bundle of 
isotropic lines $\rho_{q}:  Q(U_{q-1}^{\perp}/U_{q-1})  \longrightarrow \F \ell^{Y}_{1, \ldots, q-1}(E)$.  
The rank of the quotient bundle $U_{q-1}^{\perp}/U_{q-1}$ is $N - 2(q-1)$.  
Therefore,  the fundamental formula (\ref{eqn:FundamentalFormula(TypeBD)}) gives 
\begin{equation*} 
  \rho_{q *} (f(y_{1}, \ldots, y_{q-1}, y_{q})) 
  = [t_{q}^{N - 2q + 1}] (f(y_{1}, \ldots, y_{q-1}, t_{q}) ([2]_{\L}(t_{q}) +_{\L} z) 
                                    \mathscr{S}^{\L} ((U_{q-1}^{\perp}/U_{q})^{\vee};  1/t_{q})). 
\end{equation*} 
As with the type $C$ case, we know 
\begin{equation*} 
  \mathscr{S}^{\L} ((U_{q-1}^{\perp}/U_{q-1})^{\vee}; 1/t_{q}) 
     =    t_{q}^{-2(q-1)}  \displaystyle{\prod_{i=1}^{q-1}}  
                                                     (t_{q} +_{\L} \overline{y}_{i})  
                                                     (t_{q} +_{\L} y_{i} +_{\L}  z) 
                 \mathscr{S}^{\L}  (E^{\vee}; 1/t_{q}),  
\end{equation*} 
and hence we have 
\begin{equation*} 
\begin{array}{lll} 
 &  \rho_{q *} (f(y_{1}, \ldots, y_{q-1}, y_{q}))  \medskip \\
  & = [t_{q}^{N-1}]  \left (  f(y_{1}, \ldots, y_{q-1}, t_{q})  ([2]_{\L}(t_{q}) +_{\L} z) 
       \displaystyle{\prod_{1 \leq i \leq q-1}} (t_{q} +_{\L} \overline{y}_{i}) 
                                                          (t_{q} +_{\L} y_{i} +_{\L} z) 
                                                           \mathscr{S}^{\L} (E^{\vee};  1/t_{q}) 
                       \right ). 
\end{array} 
\end{equation*} 
 The rest of the proof is entirely analogous to the type $C$ case. 
\end{proof}

\section{Applications --Universal quadratic Schur functions--}   \label{sec:Applications} 
\subsection{Quadratic Schur functions} 
\subsubsection{Definition of  quadratic Schur functions}    \label{subsubsec:DefinitionQSF}  
In \cite[\S 4.2]{Darondeau-Pragacz2017},  Darondeau--Pragacz introduced 
the type $BCD$ analogues of  the usual Schur functions, 
called the {\it quadratic Schur functions}.  First,  we recall their definition: 
Let $E \longrightarrow X$ be a symplectic vector bundle of rank $2n$,    
where we  assume that the line bundle $L$ is  trivial.  
Then,  using  the  given symplectic form,  
we have the isomorphism $E \cong E^{\vee}$ as complex 
vector bundles, and the (cohomology) Chern roots of $E \cong E^{\vee}$ are 
given by $\pm y_{1},  \ldots, \pm y_{n}$.   
Thus,  the Segre series of $E^{\vee}$ is formally given by 
\begin{equation*} 
  s(E^{\vee}; u)  =  \prod_{j=1}^{n}  \dfrac{1}{1 - y_{j} u}  \dfrac{1}{1 + y_{j}u}  
            = \prod_{j=1}^{n}  \dfrac{1}{1 - y_{j}^{2} u^{2}}  
 = \sum_{l \geq 0}  h_{l} (\bm{y}_{n}^{2})  u^{2l},     
\end{equation*}
where $h_{l}(\bm{y}_{n}^{2}) = h_{l}(y_{1}^{2},  \ldots, y_{n}^{2})$ denotes 
the complete symmetric polynomial in $y_{1}^{2},  \ldots, y_{n}^{2}$.  
Hence,  the Segre classes of $E$  are given by  
\begin{equation*} 
     s_{2k + 1}(E^{\vee})  = 0  \quad \text{and}   \quad 
     s_{2k} (E^{\vee})  = h_{k}(\bm{y}_{n}^{2}) \;  (k = 0, 1, 2, \ldots). 
\end{equation*} 
Then, for a sequence of  integers  $I  = (I_{1},  \ldots, I_{n}) \in \Z^{n}$, 
Darondeau--Pragacz defines  the cohomology class $s^{(2)}_{I}(E^{\vee})$ to be 
\begin{equation}    \label{eqn:DefinitionQuadraticSchurFunctions}   
    s^{(2)}_{I}(E^{\vee})   := \det \, (s_{I_{i} + 2(j-i)} (E^{\vee}))_{1 \leq i, j \leq n}.  
\end{equation} 
By analogy with the usual Schur functions, 
they called 
$s^{(2)}_{I}(E^{\vee})  = s^{(2)}_{I}(\y_{n})$ the quadratic 
Schur function. 
Notice that their Segre classes $s_{i}(E)$'s are  the same as  those of our $s_{i}(E^{\vee})$. 
However, this difference does not affect the above definition because of 
$E \cong E^{\vee}$ as complex vector bundles.  
From this definition,  we observe that 
the following property holds:  
 If a partition  $I = (I_{1}, \ldots, I_{n})  \in \mathcal{P}_{n}$ has an odd part, say $I_{i}$, 
then all  entries in the  $i$-th row of the determinant $(\ref{eqn:DefinitionQuadraticSchurFunctions})$
 are zero because  $s_{\mathrm{odd}}(E^{\vee}) = 0$.  Therefore, for such a partition,  
 $s^{(2)}_{I}(E^{\vee})$ must be zero.  
Thus,  the quadratic Schur function $s^{(2)}_{I}(E^{\vee})$ is non-trivial 
only if  all  parts of $I$ are even numbers. 
  If $I$ is  such a partition,  that is, $I$ is of the form 
$2  J$ for some partition $J$,  then we see that 
\begin{equation*} 
  s^{(2)}_{I}(E^{\vee}) = \det \, (s_{I_{i} + 2(j-i)}(E^{\vee}))
        = \det \, (h_{J_{i} + (j-i)} (\bm{y}_{n}^{2}))  = s^{[2]}_{J}(E^{\vee}).  
\end{equation*} 
Here the class $s^{[2]}_{J}(E^{\vee})$ is introduced in Pragacz--Ratajski \cite[Theorem 5.13]{Pragacz-Ratajski1997}, 
and is defined as $s_{J}(\bm{y}_{n}^{2})$, the ordinary Schur polynomial  corresponding to the  partition $J$   
in the variables $\bm{y}_{n}^{2} = (y_{1}^{2},  \ldots, y_{n}^{2})$.

\subsubsection{Gysin formulas for  quadratic Schur functions} 
We shall describe the quadratic Schur functions in terms of  Gysin maps of 
type $C$ full flag bundles.   Let $\lambda  = (\lambda_{1}, \ldots, \lambda_{q})$ 
be a partition of length $\ell (\lambda) = q \leq n$.  Consider the type $C$ full flag 
bundle   $\varpi_{1, 2,  \ldots, q}:  \F \ell^{C}_{1, 2, \ldots, q}(E)   \longrightarrow X$, 
and the induced  Gysin map in cohomology, 
\begin{equation*} 
   \varpi_{1, \ldots, q *}:  H^{*}(\F \ell^{C}_{1, \ldots, q}(E))   \longrightarrow H^{*}(X). 
\end{equation*} 
Then,  the following proposition is essentially given by 
Darondeau--Pragacz \cite[Proposition 4.3]{Darondeau-Pragacz2017}: 
\begin{prop} [Characterization of the quadratic Schur functions]  \label{prop:CharacterizationQSF}
  For the Gysin map $\varpi_{1, \ldots, q *}:  H^{*}(\F \ell^{C}_{1, \ldots, q}(E))  \longrightarrow 
H^{*}(X)$, the following formula holds$:$ 
\begin{equation*} 
    \varpi_{1, \ldots, q *} (\bm{y}^{\lambda + \rho^{(2)}_{2q-1} + (2(n-q))^{q}})  
     =  \varpi_{1, \ldots, q *} \left ( \prod_{i=1}^{q} y_{i}^{\lambda_{i} + 2n - 2i + 1} \right )   
            =  s^{(2)}_{\lambda}(E^{\vee}), 
\end{equation*} 
where $\rho^{(2)}_{2q-1} := (2q-1, 2q-3, \ldots, 3, 1)$.   
\end{prop} 
\begin{proof} 
For convenience of the readers, we shall provide the proof of the proposition along the 
same lines as in \cite{Darondeau-Pragacz2017}:   
By the  D--P formula  of type $C$  \cite[Theorem 2.1]{Darondeau-Pragacz2017},  
one can compute 
\begin{equation*} 
\begin{array}{llll} 
    \varpi_{1, \ldots, q *}(\bm{y}^{\lambda + \rho^{(2)}_{2q-1} + (2(n-q))^{q}}) 
  & =   \left [ \displaystyle{\prod_{i=1}^{q}} t_{i}^{2n-1}   \right ] 
        \left ( \displaystyle{\prod_{i=1}^{q}}  t_{i}^{\lambda_{i} + 2n - 2i + 1}  
                         \prod_{1 \leq i < j \leq q}  (t_{j}^{2} - t_{i}^{2})  
                        \prod_{i=1}^{q}   s(E^{\vee};  1/t_{i})     \right )  \medskip \\
  & =    \left [ \displaystyle{\prod_{i=1}^{q}} t_{i}^{2n-1}   \right ] 
        \left (  \displaystyle{\prod_{i=1}^{q}}  t_{i}^{\lambda_{i} + 2n - 2i + 1}  
                           \det \, (t_{i}^{2j-2})_{1 \leq i , j \leq q}  
                        \prod_{i=1}^{q}   s(E^{\vee};  1/t_{i})     \right )  \medskip \\
 & =  \left  [\displaystyle{\prod_{i=1}^{q}}  t_{i}^{2n-1}  \right ]   
                       \left ( 
                                                \det \, ( t_{i}^{\lambda_{i} + 2n - 2i + 2j - 1}  s(E^{\vee};  1/t_{i}))_{1 \leq i, j \leq q}    
               \right )   \medskip   \\
  & =   \det \, ([t_{i}^{-\lambda_{i} + 2i - 2j}]   s(E^{\vee};  1/t_{i}))_{1 \leq i, j \leq q}  \medskip \\
   & =  \det \, (s_{\lambda_{i}  - 2i + 2j} (E^{\vee}))_{1 \leq i, j \leq q} 
     = s_{\lambda}^{(2)} (E^{\vee}).  
\end{array} 
\end{equation*} 
\end{proof}

Analogous with Remark \ref{rem:SymmetrizingOperatorDescription(TypeA)}, 
the Gysin map $\varpi_{1, \ldots, q *}:  H^{*}(\F \ell^{C}_{1, \ldots, q}(E))  \longrightarrow H^{*}(X)$ 
also has a symmetrizing operator description (cf. Brion \cite[Proposition 2.1]{Brion1996}):
\begin{equation*} 
    \varpi_{1, \ldots, q *}  (f)   =   \sum_{\overline{w}  \in C_{n}/C_{n-q}} 
      w \cdot \left   [     
                                \dfrac{f}{\prod_{i=1}^{q}  2y_{i}  \cdot \prod_{1 \leq i \leq q} \prod_{i  < j \leq n}  (y_{i} + y_{j}) (y_{i} - y_{j})}  
                  \right   ] 
\end{equation*} 
for $f \in H^{*}(\F \ell^{C}_{1, \ldots, q}(E))$.   Here,   $C_{n} =   (\Z/2\Z)^{n}  \rtimes S_{n}$ 
is the Weyl group  of the type $C$ root system.  
Using this formula and Proposition \ref{prop:CharacterizationQSF}, 
the quadratic Schur function $s^{(2)}_{\lambda}(E^{\vee})$ is also given by the 
following formula: 
\begin{equation}   \label{eqn:SymmetrizingOperatorDescriptionQSF}  
\begin{array}{lll} 
   s^{(2)}_{\lambda}(E^{\vee})  
 & =    \displaystyle{\sum_{\overline{w}  \in C_{n}/C_{n-q}} }
      w \cdot \left   [     
                                \dfrac{ 
                                                \bm{y}^{\lambda + \rho^{(2)}_{2q -1}  + (2(n-q))^{q}}      
                                        }
                                        {
                                                      \prod_{i=1}^{q}  2y_{i}  \cdot  
                                     \prod_{1 \leq i \leq q} \prod_{i  < j \leq n}  (y_{i} + y_{j}) (y_{i} - y_{j})}  
                  \right   ]  \medskip \\
  & =   \displaystyle{\sum_{\overline{w}  \in C_{n}/C_{n-q}} } 
      w \cdot \left   [     
                                \dfrac{ 
                                                     \prod_{i=1}^{q}  y_{i}^{\lambda_{i} + 2n - 2i + 1}
                                        }
                                        {
                                                      \prod_{i=1}^{q}  2y_{i}  \cdot  
                                     \prod_{1 \leq i \leq q} \prod_{i  < j \leq n}  (y_{i} + y_{j}) (y_{i} - y_{j})}  
                  \right   ]   \medskip    \\
& =    \dfrac{1}{2^{q}}   \displaystyle{\sum_{\overline{w}  \in C_{n}/C_{n-q}} } 
      w \cdot \left   [     
                                \dfrac{ 
                                                     \prod_{i=1}^{q}  y_{i}^{\lambda_{i} + 2(n - i)}
                                        }
                                        {  
                                     \prod_{1 \leq i \leq q} \prod_{i  < j \leq n}
                                               (y_{i}^{2} -   y_{j}^{2}) 
                                               }  
                  \right   ].     \medskip 
\end{array}  
\end{equation}

\subsection{Universal quadratic Schur functions} 
\subsubsection{Definition of  universal quadratic Schur functions}   \label{subsubsec:DefinitionUQSF}  
As the (ordinary) Schur function  has been generalized to the 
universal Schur function, we wish to consider the {\it universal analogue}
 of   the (ordinary)  quadratic Schur function.    
As we reviewed in the previous subsection, the quadratic Schur function 
due to Darondeau--Pragacz  was defined as a determinantal form (\ref{eqn:DefinitionQuadraticSchurFunctions}).  
It seems difficult to generalize  this  determinantal form directly 
 to the complex cobordism theory.  
Instead, we shall utilize  the characterization of the quadratic Schur functions 
via the Gysin map (Proposition \ref{prop:CharacterizationQSF} and (\ref{eqn:SymmetrizingOperatorDescriptionQSF})).  
Namely, we adopt the following definition:  
\begin{defn} [Universal quadratic Schur functions]  \label{defn:DefinitionUQSF}  
Let $E \longrightarrow X$ be a symplectic vector bundle of rank $2n$. 
Given a partition $\lambda   \in \mathcal{P}_{n}$ with length 
$\ell (\lambda) = q \leq n$,   consider the type $C$ full flag bundle 
$\varpi_{1, \ldots, q}: \F \ell^{C}_{1, \ldots, q}(E) \longrightarrow X$, 
and the induced Gysin map 
$\varpi_{1, \ldots, q *}: MU^{*}(\F \ell^{C}_{1, \ldots, q}(E))  \longrightarrow MU^{*}(X)$. 
 Then the universal quadratic Schur  function $s^{\L, (2)}_{\lambda}(E^{\vee})  = s^{\L, (2)}_{\lambda}(\bm{y}_{n})$ 
corresponding to $\lambda$ is defined to be 
\begin{equation}   \label{eqn:DefinitionUQSF}   
\begin{array}{lll} 
  s^{\L,  (2)}_{\lambda}(E^{\vee})  & =  s^{\L, (2)}_{\lambda}(\bm{y}_{n})  
:=  \varpi_{1, \ldots, q *} (\bm{y}^{\lambda + \rho^{(2)}_{2q-1} + (2(n-q))^{q}})  \medskip \\
& =     
 \displaystyle{\sum_{\overline{w}  \in C_{n}/C_{n-q}}} 
      w \cdot \left   [     
                                \dfrac{ 
                                               \bm{y}^{\lambda + \rho^{(2)}_{2q-1} + (2(n-q))^{q}}
                                        }
                                        {
                                                      \prod_{i=1}^{q}  [2]_{\L}(y_{i})  \cdot  
                                     \prod_{1 \leq i \leq q} \prod_{i  < j \leq n}  (y_{i} +_{\L} y_{j}) (y_{i}  +_{\L}   \overline{y}_{j})}  
                  \right   ]  \medskip \\
& = \displaystyle{\sum_{\overline{w}  \in C_{n}/C_{n-q}}} 
      w \cdot \left   [     
                                \dfrac{ 
                                                     \prod_{i=1}^{q}  y_{i}^{\lambda_{i} + 2n - 2i + 1}
                                        }
                                        {
                                                      \prod_{i=1}^{q}  [2]_{\L}(y_{i})  \cdot  
                                     \prod_{1 \leq i \leq q} \prod_{i  < j \leq n}  (y_{i} +_{\L}  y_{j}) (y_{i}  +_{\L} \overline{y}_{j})}  
                  \right   ]. 
\end{array} 
\end{equation} 
\end{defn} 
It follows immediately from the definition that under  specialization from
 $F_{\L}(u, v) = u +_{\L} v$ to $F_{a}(u, v) = u + v$, 
the function $s^{\L, (2)}_{\lambda}(E^{\vee})$ reduces to the ordinary quadratic Schur function 
$s^{(2)}_{\lambda}(E^{\vee})$.   Under  specialization from $F_{\L}(u, v) = u +_{\L} v$ to 
$F_{m}(u, v) = u \oplus v$,  we obtain the {\it $K$-theoretic 
quadratic Schur  function}, denoted  as $G^{(2)}_{\lambda}(E^{\vee}) = G^{(2)}_{\lambda}(\bm{y}_{n})$, 
which should be regarded as  a  type $BCD$ analogue of  the  Grothendieck polynomial 
$G_{\lambda}(E^{\vee}) = G_{\lambda}(\bm{y}_{n})$. 
The precise definition is given as follows:
\begin{defn}    [$K$-theoretic quadratic Schur functions]  
 In the same setting as in Definition $\ref{defn:DefinitionUQSF}$,  the 
$K$-theoretic quadratic Schur function corresponding to the 
partition $\lambda \in \mathcal{P}_{n}$ of length $q \leq n$ is defined as 
\begin{equation*}
 G^{(2)}_{\lambda}(E^{\vee}) = G^{(2)}_{\lambda}(\bm{y}_{n})
   =    \displaystyle{\sum_{\overline{w}  \in C_{n}/C_{n-q}}} 
      w \cdot \left   [     
                                \dfrac{ 
                                                     \prod_{i=1}^{q}  y_{i}^{\lambda_{i} + 2n - 2i + 1}
                                        }
                                        {
                                                      \prod_{i=1}^{q}  y_{i} \oplus y_{i}   \cdot  
                                     \prod_{1 \leq i \leq q} \prod_{i  < j \leq n}  (y_{i}  \oplus   y_{j}) (y_{i} \ominus y_{j})}  
                  \right   ]. 
\end{equation*} 
\end{defn} 

\subsubsection{
Gysin formulas for universal quadratic Schur functions 
} 
By Definition \ref{defn:DefinitionUQSF}, the analogue of  \cite[Proposition 4.7]{Darondeau-Pragacz2017}
 in complex cobordism can be easily obtained: 
Let $E  \longrightarrow X$ be a symplectic vector bundle of rank $2n$, 
where we assume that the line bundle  $L$ is trivial.  
Consider the isotropic Grassmann bundle $\varpi_{q}:  \F \ell^{C}_{q}(E)   \longrightarrow X$
for $q = 1, \ldots, n$.  The induced Gysin map 
$\varpi_{q *}:  MU^{*}(\F \ell^{C}_{q} (E))  \longrightarrow MU^{*}(X)$
has the following symmetrizing operator description: 
\begin{equation*} 
  \varpi_{q *} (f)   =  \sum_{\overline{w} \in C_{n}/S_{q} \times C_{n-q}} 
              w  \cdot \left [ 
                                 \dfrac{f}{ 
                                                \prod_{i=1}^{q}  [2]_{\L}(y_{i})  
                                         \cdot  \prod_{\substack{1 \leq i \leq q, \\                                                                     
                                                                               q + 1 \leq j \leq n}
                                                              }  
                                                        (y_{i} +_{\L}  \overline{y}_{j})   
                                          \cdot \prod_{\substack{1 \leq i \leq q, \\
                                                                              i < j \leq n} 
                                                         } 
                                                           (y_{i} +_{\L}  y_{j})    
                                        } 
                            \right ] 
\end{equation*} 
for $f \in  MU^{*}(\F \ell^{C}_{q} (E))$.  
Next, we need to define the universal quadratic Schur functions 
for arbitrary sequences of positive integers.  Namely, for a sequence of positive 
integers $I = (I_{1}, \ldots, I_{q})  \in  (\Z_{> 0})^{q}$ (in this case, we say that 
the length $\ell (I)$ of $I$ is $q$), the corresponding universal quadratic Schur 
function $s^{\L, (2)}_{I}(E^{\vee})  = s^{\L, (2)}_{I}(\bm{y}_{n})$ 
is defined by the same expression as  (\ref{eqn:DefinitionUQSF}).  
With these preparations,  one can obtain the following result:
\begin{prop}     \label{prop:GeneralizationPragacz-RatajskiFormula(ComplexCobordism)} 
For  a sequence of non-negative integers
 $I  = (I_{1}, \dots, I_{q})  \in (\Z_{\geq 0})^{q}$, which satisfies 
the condition 
$I_{i}  >  2n - q - i + 1$ for $i = 1, \ldots, q$,  one has the 
following Gysin formula$:$
\begin{equation}   \label{eqn:FormulaUsingSchurFunctions(TypeC)}    
   \varpi_{q *}  (s^{\L}_{I}(U_{q}^{\vee}))   =  s^{\L, (2)}_{I - \rho_{q}}(E^{\vee}), 
\end{equation} 
where  $\rho_{q} =  (2n - q, 2n - q -1, \ldots, 2n - q   - i + 1, \ldots, 2n  - 2q + 1)$.   
\end{prop} 
\begin{proof} 
By (\ref{eqn:DefinitionUSF}) and  the symmetrizing operator description 
of $\varpi_{q *}$ mentioned above, one can compute 
\begin{equation*} 
\begin{array}{llll} 
   \varpi_{q *}(s^{\L}_{I}(U_{q}^{\vee}))  
 & =  \varpi_{q *} (s^{\L}_{I}(\bm{y}_{q}))  \medskip \\
 & =   \displaystyle{\sum_{\overline{w}  \in C_{n}/S_{q} \times C_{n-q}}} 
                  w \cdot \left [ 
                                 \dfrac{
                                               \sum_{v \in S_{q}}  
                                                      v \cdot \left [ 
                                                      \dfrac{\prod_{i=1}^{q} y_{i}^{I_{i} + q-i}} 
                                                              { \prod_{1 \leq i < j \leq q} (y_{i} +_{\L} \overline{y}_{j})    
                                                             } 
                                                                  \right ] 
                                             }{ 
                                                \prod_{i=1}^{q}  [2]_{\L}(y_{i})  
                                         \cdot  \prod_{\substack{1 \leq i \leq q, \\                                                                     
                                                                               q + 1 \leq j \leq n}
                                                              }  
                                                        (y_{i} +_{\L}  \overline{y}_{j})   
                                          \cdot \prod_{\substack{1 \leq i \leq q, \\
                                                                              i < j \leq n} 
                                                         } 
                                                           (y_{i} +_{\L}  y_{j})    
                                        } 
                            \right ].                \medskip \\
\end{array} 
\end{equation*} 
Since $\prod_{i=1}^{q} [2]_{\L} (y_{i})$, $\prod_{\substack{1 \leq i \leq q, \\
                                                                              q + 1 \leq j \leq n} 
                                                               } (y_{i} +_{\L}  \overline{y}_{j})$, 
and $\prod_{\substack{1 \leq i \leq q, \\
                                 i < j \leq n} 
                }  (y_{i} +_{\L}  y_{j})$ 
are all $S_{q}$-invariant,  the above expression is equal to 
\begin{equation*} 
\begin{array}{llll} 
 &  \displaystyle{\sum_{\overline{w} \in C_{n}/S_{q} \times C_{n-q}}}  
     \sum_{v \in S_{q}}  
            wv \cdot \left [ 
                                    \dfrac{ \prod_{i=1}^{q} y_{i}^{I_{i} + q-i}  } 
                                      {   
                                        \prod_{i=1}^{q} [2]_{\L} (y_{i})  
                                        \cdot  \prod_{1 \leq i \leq q} 
                                                  \prod_{i < j \leq n} 
                                                  (y_{i} +_{\L} \overline{y}_{j}) (y_{i} +_{\L} y_{j})           
                                      }                     
                            \right ]    \medskip \\
  & =  \displaystyle{\sum_{\overline{u} \in C_{n}/C_{n-q}}}  
           u \cdot \left [ 
                                    \dfrac{ \prod_{i=1}^{q} y_{i}^{(I_{i} - 2n + q + i-1)  + 2n - 2i + 1}  } 
                                      {   
                                        \prod_{i=1}^{q} [2]_{\L} (y_{i})  
                                        \cdot  \prod_{1 \leq i \leq q} 
                                                  \prod_{i < j \leq n} 
                                                  (y_{i} +_{\L} \overline{y}_{j}) (y_{i} +_{\L} y_{j})           
                                      }                     
                            \right ]    \medskip \\
  & =  s^{\L, (2)}_{I  - \rho_{q}} (E^{\vee}) 
\end{array} 
\end{equation*} 
as desired. 
\end{proof}  
As a special case of the above proposition ($q = n$), one  immediately obtains
a complex cobordism analogue of  the result of 
 Pragacz--Ratajski  \cite[Theorem 5.13]{Pragacz-Ratajski1997}:  
\begin{cor} [Pragacz--Ratajski formula in complex cobordism]  
Let $I = (I_{1},  \ldots, I_{n})  \in (\Z_{\geq 0})^{n}$ be a sequence of 
non-negative integers satisfying $I_{i} > n - i + 1 \; (i = 1, \ldots, n)$. 
Then, the following Gysin formula holds for the Lagrangian Grassmann bundle
$\varpi_{n}:   LG_{n}(E)  = \F \ell_{n}^{C}(E)  \longrightarrow X$$:$  
\begin{equation}  \label{eqn:Pragacz-RatajskiTheorem(ComplexCobordism)}   
    \varpi_{n *}  (s^{\L}_{I}(U_{n}^{\vee}))  = s^{\L, (2)}_{I - \rho_{n}} (E^{\vee}). 
\end{equation} 
\end{cor} 
If we reduce the universal formal group law $F_{\L}(u, v) = u +_{\L} v$ 
to  $F_{a}(u, v) = u + v$,  then the formula (\ref{eqn:Pragacz-RatajskiTheorem(ComplexCobordism)}) 
reduces to 
\begin{equation} \label{eqn:Pragacz-RatajskiTheorem(Cohomology)}   
   \varpi_{n *} (s_{I}(U_{n}^{\vee}))  =  s^{(2)}_{I - \rho_{n}} (E^{\vee}). 
\end{equation} 
If one of the parts of $I - \rho_{n}$ is odd,  then the right-hand side of
 (\ref{eqn:Pragacz-RatajskiTheorem(Cohomology)}) is zero.  
Therefore,  the element $s_{I}(U_{n}^{\vee})$ 
has a nonzero image under $\varpi_{n *}$ only if $I$ is of the form 
$2J + \rho_{n}$ for some $J_{n}  \in (\Z_{\geq 0})^{n}$.  
If $I = 2J + \rho_{n}$, then one obtains 
\begin{equation*} 
     \varpi_{n *} (s_{I}(U_{n}^{\vee}))  =  s^{(2)}_{2J}(E^{\vee})  = s^{[2]}_{J}(E^{\vee})
\end{equation*} 
as observed in  \S  \ref{subsubsec:DefinitionQSF}.  
This is the original Pragacz--Ratajski formula.

\subsubsection{Generating function for  universal quadratic Schur functions} 
As an application of  our  Darondeau--Pragacz formulas in complex cobordism, 
one can derive the generating function for the universal quadratic Schur functions. 
According to Theorem \ref{thm:TypeCD-PFormula(ComplexCobordism)},    
the D--P formula for  the Gysin map 
$\varpi_{1, \ldots, q *}:  MU^{*}(\F \ell^{C}_{1, \ldots, q}(E))  \longrightarrow MU^{*}(X)$
takes the following form: 
For a polynomial $f(t_{1}, \ldots, t_{q})  \in MU^{*}(X)[t_{1}, \ldots, t_{q}]$, 
one has 
\begin{equation*} 
\begin{array}{lll} 
   \varpi_{1, \ldots, q *}(f(y_{1}, \ldots, y_{q}))  
  &  =   \left [  \displaystyle{\prod_{i=1}^{q}}  t_{i}^{2n-1}  \right ] 
        \left (  
                        f(t_{1}, \ldots, t_{q})  \displaystyle{\prod_{1 \leq i < j \leq q}} (t_{j} +_{\L} \overline{t}_{i}) 
                                            (t_{j} +_{\L}  t_{i})  
         \right.   \medskip \\
   &  \hspace{7cm} \left.   \times  \;  \displaystyle{\prod_{i=1}^{q}}  
                                            \mathscr{S}^{\L} (E^{\vee};  1/t_{i})  
                      \right ).   \medskip 
\end{array} 
\end{equation*} 
This formula, together with Definition \ref{defn:DefinitionUQSF}, 
yields the 
following expression: 
\begin{equation*} 
\begin{array}{llll} 
    s^{\L,  (2)}_{\lambda} (E^{\vee})  
      & =  \varpi_{1, \ldots, q *} (\bm{y}^{\lambda + \rho^{(2)}_{2q-1} + (2(n-q))^{q}})  \medskip \\
    &  =     \left [ \displaystyle{\prod_{i=1}^{q}} t_{i}^{2n-1}   \right ] 
        \left ( \displaystyle{\prod_{i=1}^{q}}  t_{i}^{\lambda_{i} + 2n - 2i + 1}  
                         \prod_{1 \leq i < j \leq q}  (t_{j} +_{\L} \overline{t}_{i}) (t_{j} +_{\L} t_{i})   
                        \prod_{i=1}^{q}   \mathscr{S}^{\L} (E^{\vee};  1/t_{i})     \right )  \medskip \\
   & =  \left [\displaystyle{\prod_{i=1}^{q}}  t_{i}^{-\lambda_{i}}   \right ]  
            \left (    \displaystyle{  
                                           \prod_{1 \leq i < j \leq q}
                                         }     \dfrac{t_{j} +_{\L} \overline{t}_{i}}{t_{j}} \cdot 
                                               \dfrac{t_{j} +_{\L}  t_{i}}{t_{j}}   
                                            \prod_{i=1}^{q}  \mathscr{S}^{\L} (E^{\vee};  1/t_{i})  
        \right ).    \medskip 
\end{array} 
\end{equation*} 
Since the $MU^{*}$-theory Chern roots of $E^{\vee}$ are given by $y_{i},  \overline{y}_{i} \; (1 \leq i \leq n)$,  from (\ref{eqn:SegreSeries(ComplexCobordism)}),  
the Segre series of $E^{\vee}$ is given as follows: 
\begin{equation*} 
   \mathscr{S}^{\L} (E^{\vee}; u)  = \left.  
                                       \dfrac{1}{\mathscr{P}^{\L} (z)}  
                                      \dfrac{z^{2n}} { \prod_{j=1}^{n} (z +_{\L}  y_{j}) (z +_{\L}  \overline{y}_{j}) } \right |_{z = u^{-1}}. 
\end{equation*} 
Let us reformulate the above result in terms of symmetric functions: 
For   independent variables $\bm{y}_{n} = (y_{1}, \ldots, y_{n})$, we set 
\begin{equation*} 
\begin{array}{rll} 
     s^{\L, (2)}  (u)  & :=    
                                       \dfrac{1}{\mathscr{P}^{\L} (u)}  
                                      \dfrac{u^{2n}} { \prod_{j=1}^{n} (u +_{\L}  y_{j}) (u +_{\L}  \overline{y}_{j}) },  \medskip \\
     s^{\L, (2)} (u_{1}, \ldots, u_{q}) &  := \displaystyle{\prod_{i=1}^{q}}  s^{\L, (2)}(u_{i})
                                                 \displaystyle{  
                                           \prod_{1 \leq i < j \leq q}
                                         }     \dfrac{t_{j} +_{\L} \overline{t}_{i}}{t_{j}} \cdot 
                                               \dfrac{t_{j} +_{\L}  t_{i} }{t_{j}}.  \medskip 
\end{array}     
\end{equation*} 
Then,   we have the following result: 
\begin{theorem}   \label{thm:GFUQSF} 
For a partition $\lambda = (\lambda_{1}, \ldots, \lambda_{q})$ of length $\ell (\lambda) 
= q \leq n$,  the  universal quadratic Schur function $s^{\L, (2)}(\bm{y}_{n})$ 
is the coefficient of  $u_{1}^{-\lambda_{1}}  \cdots u_{q}^{-\lambda_{q}}$ 
in $s^{\L, (2)}(u_{1},  \ldots, u_{q})$, that is, 
\begin{equation*} 
     s^{\L, (2)}_{\lambda}(\bm{y}_{n}) =  [u_{1}^{-\lambda_{1}}  \cdots  u_{q}^{-\lambda_{q}}] 
                                          (s^{\L, (2)}(u_{1},  \ldots, u_{q})). 
\end{equation*} 
\end{theorem} 
If we specialize  $F_{\L}(u, v) = u +_{\L} v$ to $F_{a}(u, v) = u + v$, then 
the factor $\displaystyle{\prod_{1 \leq i < j \leq r}}  \dfrac{t_{j} +_{\L} \overline{t}_{i}}{t_{j}} 
   \cdot \dfrac{t_{j} +_{\L}  t_{i} }{t_{j}}$ reduces to 
\begin{equation*} 
   \prod_{1 \leq i < j \leq q}  \dfrac{t_{j} - t_{i}} {t_{j}} \cdot  \dfrac{ t_{j} + t_{i}}{t_{j}}  
  = \prod_{1 \leq i < j \leq q}  \left  ( 1 - t_{i}^{2}/t_{j}^{2} \right )  
  = \det \, (t_{i}^{2(j-i)} )_{1 \leq i, j \leq q}. 
\end{equation*} 
Therefore, by Theorem \ref{thm:GFUQSF}, the determinantal formula (\ref{eqn:DefinitionQuadraticSchurFunctions})
  can be recovered.

In $K$-theory, we can say more: 
Let $E$ be a symplectic vector bundle of rank $2n$ with $K$-theoretic Chern roots 
$y_{i}, \overline{y}_{i}  = \ominus y_{i} \; (i = 1, \ldots, n)$.   Then, 
by (\ref{eqn:SegreSeries(ComplexCobordism)}), the $K$-theoretic 
Segre series $G(E^{\vee}; 1/u)$ is given by  
\begin{equation*}  
 G(E^{\vee}; 1/u)   = 
                                       \dfrac{1}{1 - \beta u}  
                                      \dfrac{u^{2n}} { \prod_{j=1}^{n} (u \oplus y_{j}) (u  \ominus y_{j}) }. 
\end{equation*} 
Then,  by Theorem \ref{thm:GFUQSF}, we have 
\begin{equation}  \label{eqn:G^{(2)}_lambda(E)}   
    G^{(2)}_{\lambda} (E^{\vee})  
    =  \left [\displaystyle{\prod_{i=1}^{q}}  t_{i}^{-\lambda_{i}}   \right ]  
            \left (    \displaystyle{  
                                           \prod_{1 \leq i < j \leq q}
                                         }     \dfrac{t_{j}  \ominus t_{i}}{t_{j}} \cdot 
                                              \dfrac{t_{j}  \oplus t_{i}}{t_{j}}   
                                            \prod_{i=1}^{q}  G (E^{\vee};  1/t_{i})  
        \right ).   
\end{equation}
Notice that the following identity holds:  
\begin{equation*} 
  \prod_{1 \leq i < j \leq q}     \dfrac{t_{j}  \ominus t_{i}}{t_{j}} 
                                        \cdot  \dfrac{t_{j}  \oplus t_{i} }{t_{j}}   
  = \prod_{1 \leq i <j \leq q}   \left ( 1  -    \dfrac{t_{i}^{2}}{1 - \beta t_{i}} \left /
                                                          \dfrac{t_{j}^{2}}{1 - \beta t_{j}} \right.   
                                          \right ) 
  = \det \,  ((1 - \beta t_{i})^{i-j}  t_{i}^{2(j-i)} )_{1 \leq i, j \leq q}.  
\end{equation*} 
Therefore, by (\ref{eqn:G^{(2)}_lambda(E)}), we have the following 
result:\footnote{
An analogous   determinantal formula for the  Grothendieck 
polynomials was recently  obtained by Hudson--Ikeda--Matsumura--Naruse \cite[Theorem 3.13]{HIMN2017}.  
}
\begin{theorem}   [Determinantal formula for the $K$-theoretic quadratic Schur function]  \label{thm:DeterminantalFormulaKQSF}   
\begin{equation}    
   G^{(2)}_{\lambda}(E^{\vee})  =    
     \det \, \left (  \sum_{k=0}^{\infty}  \binom{i - j} {k}  (-\beta)^{k} 
                      G_{\lambda_{i} + 2(j-i) + k}  (E^{\vee})  
                \right )_{1 \leq i, j \leq q}.  
\end{equation}  
\end{theorem} 

\section{Appendix} 
\subsection{Quillen's Residue Formula}   \label{subsec:QuillenResidueFormula}  
As we mentioned in \S  \ref{subsubsec:FundamentalGysinFormulaProjectiveBundle(ComplexCobordism)}, 
we proceed with the calculation  of Quillen's residue formula, i.e.,  
the right-hand side of (\ref{eqn:QuillenResidueFormula}) in the following manner:  
First, we   consider the case where $f(t) = t^{N}$,  
a monomial in $t$ of degree $N \geq 0$.    
 We expand $\mathscr{P}^{\L}(t, y_{j})$ as 
$\sum_{\ell_{j} = 0}^{\infty}  \mathscr{P}_{\ell_{j}} (y_{j}) t^{\ell_{j}}$ 
for $j = 1, \ldots, n$.   Then,   we compute 
\begin{equation*} 
\begin{array}{lll} 
          \dfrac{t^{N}}{\mathscr{P}^{\L} (t)  \prod_{j=1}^{n} (t +_{\L} \overline{y}_{j})}  
  & =  t^{N}  \times  \dfrac{1}{\mathscr{P}^{\L}(t)}  
       \times \dfrac{1}{ \prod_{j=1}^{n}  \dfrac{t - y_{j}} {\mathscr{P}^{\L} (t, y_{j})} } \medskip \\
  & =  t^{N - n}    \times  \dfrac{1}{\mathscr{P}^{\L} (t)}  
        \times   \displaystyle{\prod_{j=1}^{n}}  \mathscr{P}^{\L} (t, y_{j})  \times \prod_{j=1}^{n}  \dfrac{1}{1 - y_{j}t^{-1}}   \medskip \\
  & =  t^{N-n}   \times 
           \left (  \displaystyle{\sum_{k=0}^{\infty}} [\C P^{k}] t^{k} \right )  
            \times \displaystyle{\prod_{j=1}^{n}}  
           \left  ( \sum_{\ell_{j} = 0}^{\infty}  \mathscr{P}_{\ell_{j}} (y_{j}) t^{\ell_{j}} \right ) 
          \times \left ( \sum_{m=0}^{\infty}  h_{m} (\bm{y}_{n})  t^{-m}  \right ).   \medskip
\end{array} 
\end{equation*} 
For brevity, we put 
\begin{equation*} 
   \prod_{j=1}^{n}   \left  ( \sum_{\ell_{j} = 0}^{\infty}  \mathscr{P}_{\ell_{j}} (y_{j}) t^{\ell_{j}} \right ) 
   =  \sum_{\ell = 0}^{\infty}  
      \left (\sum_{ \substack{\ell_{1} + \cdots + \ell_{n} = \ell  \\
                        \ell_{1} \geq 0, \ldots, \ell_{n} \geq 0}  
                     }   \prod_{j=1}^{n} \mathscr{P}_{\ell_{j}}(y_{j}) 
              \right ) t^{\ell} 
  =  \sum_{\ell = 0}^{\infty}  \mathscr{P}_{\ell}(\bm{y}_{n})  t^{\ell}. 
\end{equation*} 
Then,   the computation continues as 
\begin{equation*} 
\begin{array}{lll} 
  &  t^{N-n}  \times  
        \left (  \displaystyle{\sum_{k=0}^{\infty}} [\C P^{k}] t^{k} \right )  
          \times   \left \{   \displaystyle{\sum_{r= -\infty}^{\infty}} 
                      \left (  \sum_{\ell = r}^{\infty}  \mathscr{P}_{\ell}(\bm{y}_{n})  h_{\ell - r}(\bm{y}_{n})  \right ) t^{r}        
                 \right \}            \medskip \\
     =  &  t^{N - n}  \times \left (  \displaystyle{\sum_{k=0}^{\infty}} [\C P^{k}] t^{k} \right )  
       \times \left \{  
                      \displaystyle{\sum_{r=0}^{\infty}} 
                        \left ( 
                              \sum_{\ell = r}^{\infty}   
                                    \mathscr{P}_{\ell} (\bm{y}_{n})  h_{\ell - r} (\bm{y}_{n}) 
                     \right ) t^{r} 
              +    \displaystyle{\sum_{r=1}^{\infty}}  
                            \left ( 
                                      \sum_{\ell = 0}^{\infty}  \mathscr{P}_{\ell}(\bm{y}_{n}) h_{\ell + r}(\bm{y}_{n})  \right )  t^{-r}   
                   \right \}   \medskip   \\
  =  &  t^{N - n}  \times  
         \left [   \displaystyle{\sum_{s = 0}^{\infty}} 
                 \left \{   \sum_{k= 0}^{s}  \sum_{\ell = s - k}^{\infty}  
                               [\C P^{k}]  \mathscr{P}_{\ell} (\bm{y}_{n})  h_{k + \ell - s} (\bm{y}_{n})  
                      \right \}   t^{s}     \right.    \medskip \\
  &   \hspace{6cm}    \left.              +  \displaystyle{\sum_{s = -\infty}^{\infty}} 
                         \left \{   \sum_{k = s + 1}^{\infty}  \sum_{\ell = 0}^{\infty}  [\C P^{k}]  \mathscr{P}_{\ell}(\bm{y}_{n})  h_{k +\ell - s} (\bm{y}_{n})    \right \}   t^{s}  \right ]   \medskip  \\
=  &  t^{N - n}  \times  
         \left [   \displaystyle{\sum_{s = 0}^{\infty}} 
                 \left \{   \sum_{k= 0}^{s}  \sum_{\ell = s - k}^{\infty}  
                               [\C P^{k}]  \mathscr{P}_{\ell} (\bm{y}_{n})  h_{k + \ell - s} (\bm{y}_{n})  
                      \right \}   t^{s}     \right.    \medskip \\           
      + &  \left.  \displaystyle{\sum_{s = 0}^{\infty}} 
                         \left \{   \sum_{k = s + 1}^{\infty}  \sum_{\ell = 0}^{\infty}  [\C P^{k}]  \mathscr{P}_{\ell}(\bm{y}_{n})  h_{k +\ell - s} (\bm{y}_{n})    \right \}   t^{s}   
     +      \displaystyle{\sum_{s = 1}^{\infty}} 
                         \left \{   \sum_{k =  0}^{\infty}  \sum_{\ell = 0}^{\infty}  [\C P^{k}]  \mathscr{P}_{\ell}(\bm{y}_{n})  h_{k +\ell +  s} (\bm{y}_{n})    \right \}   t^{-s}        \right ].  \medskip  \\
\end{array} 
\end{equation*} 
Next,  we extract the coefficient of $t^{-1}$  in the above formal Laurent series.  
In the case  of $N \geq n$,  one sees immediately  that 
the coefficient of $t^{-1}$ is 
\begin{equation*} 
      \sum_{k =  0}^{\infty}  \sum_{\ell = 0}^{\infty}
        [\C P^{k}]  \mathscr{P}_{\ell}(\bm{y}_{n})  h_{k +\ell +  N - n + 1} (\bm{y}_{n}). 
\end{equation*} 
In the case of  $0 \leq N < n$,  the coefficient of $t^{-1}$  is  
\begin{equation*} 
\begin{array}{lll} 
  &      \displaystyle{\sum_{k= 0}^{n - N - 1}}  \sum_{\ell = n - N - 1 - k}^{\infty}  
                               [\C P^{k}]  \mathscr{P}_{\ell} (\bm{y}_{n})  h_{k + \ell + N - n + 1} (\bm{y}_{n})  
+  \sum_{k = n - N }^{\infty}  \sum_{\ell = 0}^{\infty}  [\C P^{k}]  \mathscr{P}_{\ell}(\bm{y}_{n})  h_{k +\ell + N - n + 1} (\bm{y}_{n})    \medskip \\
  & =    \displaystyle{\sum_{k= 0}^{\infty}}  \sum_{\ell = n - N - 1 - k}^{\infty}  
                               [\C P^{k}]  \mathscr{P}_{\ell} (\bm{y}_{n})  h_{k + \ell + N - n + 1} (\bm{y}_{n}).    
\end{array} 
\end{equation*} 
Summing up the above calculation, we get the following result: 
\begin{equation}   \label{eqn:QuillenResidueFormulaCalculation}   
   \underset{t = 0}{\mathrm{Res}'} 
          \dfrac{t^{N}} {\mathscr{P}^{\L} (t)  \prod_{j=1}^{n} (t +_{\L}  \overline{y}_{j}) } 
   =  \left \{     
           \begin{array}{lll} 
                  &  \displaystyle{\sum_{k =  0}^{\infty}}  \sum_{\ell = 0}^{\infty}
        [\C P^{k}]  \mathscr{P}_{\ell}(\bm{y}_{n})  h_{k +\ell +  N - n + 1} (\bm{y}_{n})     \quad  (N \geq n),  \medskip \\
                  &   \displaystyle{\sum_{k= 0}^{\infty}}  \sum_{\ell = n - N - 1 - k}^{\infty}  
                               [\C P^{k}]  \mathscr{P}_{\ell} (\bm{y}_{n})  h_{k + \ell + N - n + 1} (\bm{y}_{n})     \quad   (0 \leq N  < n).  \medskip 
            \end{array} 
      \right. 
\end{equation} 
For a general polynomial of the form $f(t) = \sum_{N=0}^{M} a_{N} t^{N}  \in MU^{*}(X)[t]$, 
one has   $\varpi_{1 *} (f(y_{1}))  = \varpi_{1 *} \left  (\sum_{N=0}^{M} a_{N} y_{1}^{N}  \right )  
           = \sum_{N=0}^{M} a_{N} \varpi_{1 *} (y_{1}^{N})$  
since  the Gysin map $\varpi_{1 *}$ is an $MU^{*}(X)$-homomorphism.  
From this, we can calculate $\varpi_{1 *}(f(y_{1}))$.  


\end{document}